\documentclass[10pt,righttag]{amsart}
\usepackage{amssymb,verbatim,amscd,mathrsfs,stmaryrd,wasysym,latexsym, bbm}
\usepackage[all]{xy} 
\usepackage{cancel,color}
\usepackage{lscape}
\begin{document}
\textheight= 195 mm
\textwidth = 125 mm

\newcommand{\V}{{\mathcal V}}      
\renewcommand{\O}{{\mathcal O}}
\newcommand{\LL}{\mathcal L}
\newcommand{\Ext}{\hbox{\rm Ext}}
\newcommand{\Tor}{\hbox{\rm Tor}}
\newcommand{\Hom}{\hbox{Hom}}
\newcommand{\Proj}{\hbox{Proj}}
\newcommand{\GrMod}{\hbox{GrMod}}
\newcommand{\grmod}{\hbox{gr-mod}}
\newcommand{\Tors}{\hbox{Tors}}
\newcommand{\gr}{\hbox{gr}}
\newcommand{\tors}{\hbox{tors}}
\newcommand{\rank}{\hbox{rank}}
\newcommand{\End}{\hbox{{\rm End}}}
\newcommand{\Der}{\hbox{Der}}
\newcommand{\GKdim}{\hbox{GKdim}}
\newcommand{\gldim}{\hbox{gldim}}
\newcommand{\im}{\hbox{im}}
\renewcommand{\ker}{\hbox{ker}}
\def\bee{\begin{eqnarray}}
\def\eee{\end{eqnarray}}
\newcommand{\coker}{{\rm coker}}

\newcommand{\lonto}{{\protect \longrightarrow\!\!\!\!\!\!\!\!\longrightarrow}}

\renewcommand{\c}{\cancel}
\newcommand{\h}{\frak h}
\newcommand{\fp}{\frak p}
\newcommand{\m}{{\mu}}
\newcommand{\gl}{{\frak g}{\frak l}}
\newcommand{\ssl}{{\frak s}{\frak l}}
\newcommand{\tw}{{\rm tw}}

\newcommand{\ds}{\displaystyle}
\newcommand{\s}{\sigma}
\renewcommand{\l}{\lambda}
\renewcommand{\a}{\alpha}
\renewcommand{\b}{\beta}
\newcommand{\G}{\Gamma}
\newcommand{\g}{\gamma}
\newcommand{\z}{\zeta}
\newcommand{\e}{\epsilon}
\renewcommand{\d}{\delta}
\newcommand{\p}{\rho}
\renewcommand{\t}{\tau}
\newcommand{\n}{\nu}
\newcommand{\x}{\chi}
\newcommand{\w}{\omega}

\newcommand{\A}{{\Bbb A}}
\newcommand{\C}{{\Bbb C}}
\newcommand{\N}{{\Bbb N}}
\newcommand{\Z}{{\Bbb Z}}
\newcommand{\ZZ}{{\Bbb Z}}
\newcommand{\Q}{{\Bbb Q}}
\renewcommand{\k}{\mathbb K}

\newcommand{\E}{{\mathcal E}}
\newcommand{\K}{{\mathcal K}}
\renewcommand{\S}{{\mathcal S}}
\newcommand{\T}{{\mathcal T}}

\newcommand{\GL}{{GL}}

\newcommand{\rowxy}{(x\ y)}
\newcommand{\colxy}{ \left({\begin{array}{c} x \\ y \end{array}}\right)}
\newcommand{\scolxy}{\left({\begin{smallmatrix} x \\ y
\end{smallmatrix}}\right)}

\renewcommand{\P}{{\Bbb P}}

\newcommand{\la}{\langle}
\newcommand{\ra}{\rangle}
\newcommand{\tensor}{\otimes}
\newcommand{\tsr}{\tensor}
\newcommand{\ol}{\overline}

\newtheorem{thm}{Theorem}[section]
\newtheorem{lemma}[thm]{Lemma}
\newtheorem{cor}[thm]{Corollary}
\newtheorem{prop}[thm]{Proposition}

\theoremstyle{definition}
\newtheorem{defn}[thm]{Definition}
\newtheorem{notn}[thm]{Notation}
\newtheorem{ex}[thm]{Example}
\newtheorem{rmk}[thm]{Remark}
\newtheorem{rmks}[thm]{Remarks}
\newtheorem{note}[thm]{Note}
\newtheorem{example}[thm]{Example}
\newtheorem{problem}[thm]{Problem}
\newtheorem{ques}[thm]{Question}
\newtheorem{thingy}[thm]{}

\newcommand{\onto}{{\protect \rightarrow\!\!\!\!\!\rightarrow}}
\newcommand{\donto}{\put(0,-2){$|$}\put(-1.3,-12){$\downarrow$}{\put(-1.3,-14.5) 

{$\downarrow$}}}

\newcounter{letter}
\renewcommand{\theletter}{\rom{(}\alph{letter}\rom{)}}

\newenvironment{lcase}{\begin{list}{~~~~\theletter} {\usecounter{letter}
\setlength{\labelwidth4ex}{\leftmargin6ex}}}{\end{list}}

\newcounter{rnum}
\renewcommand{\thernum}{\rom{(}\roman{rnum}\rom{)}}

\newenvironment{lnum}{\begin{list}{~~~~\thernum}{\usecounter{rnum}
\setlength{\labelwidth4ex}{\leftmargin6ex}}}{\end{list}}

\thispagestyle{empty}

\title[Classification of certain graded twisted tensor products]{Classification, koszulity and Artin-Schelter regularity of certain graded twisted tensor products}

\keywords{Koszul algebras, quadratic algebras, twisted tensor products, Artin-Schelter regular algebras}

\author[  Conner, Goetz ]{ }

  \subjclass[2010]{16S37, 16S38}
\maketitle

\begin{center}

\vskip-.2in Andrew Conner \\
\bigskip

Department of Mathematics and Computer Science\\
Saint Mary's College of California\\
Moraga, CA 94575\\
\bigskip

 Peter Goetz \\
\bigskip

Department of Mathematics\\ Humboldt State University\\
Arcata, California  95521
\\ \ \\

\end{center}

\setcounter{page}{1}

\thispagestyle{empty}

\vspace{0.2in}

\begin{abstract}

Let $\k$ be an algebraically closed field. We classify all of the quadratic twisted tensor products $A \tsr_{\t} B$ in the cases where $(A, B) = (\k[x], \k[y])$ and $(A, B) = (\k[x, y], \k[z])$. We determine when a quadratic twisted tensor product of this form is Koszul, and when it is Artin-Schelter regular.
\end{abstract}

\bigskip

\section{Introduction}

Let $\k$ be an algebraically closed field, and let $A$ and $B$ be associative $\k$-algebras. In \cite{Cap}, C\v{a}p, Schichl, and Van\v{z}ura introduced a very general notion of product for $A$ and $B$ called a \emph{twisted tensor product}, denoted $A\tsr_{\t} B$. (For precise definitions of terminology we refer the reader to Section 2.) Efforts to understand how ring-theoretic and homological properties behave with respect to this product have been the source of several recent papers (see, for example, \cite{C-G}, \cite{GKM}, \cite{JPS}, \cite{SZL}, \cite{ShepWi2}). In the article \cite{C-G}, we initiated a detailed study of the Koszul property for graded twisted tensor products. The current paper grew out of our efforts to address two problems left unresolved in \cite{C-G}. 

\begin{problem}\label{quad example}
If $A$ and $B$ are Koszul algebras and $A\tsr_{\t} B$ is quadratic, must $A\tsr_{\t} B$ be a Koszul algebra?
\end{problem}

\begin{problem}\label{classification}
Classify quadratic twisted tensor products of $\k[x]$ and $\k[y]$ up to isomorphism (of twisted tensor products).
\end{problem}

Our solutions to these problems are related. By \cite[Proposition 5.5]{C-G}, all quadratic twisted tensor products of $\k[x]$ and $\k[y]$ are Koszul. By \cite[Theorem 5.3]{C-G}, the same is true if $\k[x]$ or $\k[y]$ is replaced by $\k[x]/\la x^2\ra$ or $\k[x]/\la y^2\ra$. Hence a negative answer to Problem \ref{quad example} would require $A\tsr_{\t} B$ to have at least three algebra generators. This requirement is the impetus for our study, and subsequent classification, of quadratic twisted tensor products of $\k[x,y]$ and $\k[z]$. A significant part of this classification reduces to Problem \ref{classification}.

In \cite[Section 6]{C-G}, we partially settled Problem \ref{classification}, but we were unable to handle one case. Our partial result described possible isomorphism types of $\k[x]\tsr_{\t}\k[y]$ in terms of a one-parameter family of algebras. The parameter could take on any value, except the zeros of an interesting family of polynomials related to the Catalan numbers. The appearance of these polynomials motivated us to completely resolve Problem \ref{classification}, which we do in Section 3. Our first result is the following.

\begin{thm}[Proposition \ref{oneSidedTwoGen} and Theorem \ref{quadratic ttps of polynomial rings of one variable}]
\label{introClassification2D}
Every quadratic twisted tensor product of $\k[x]$ and $\k[y]$ is isomorphic to either 
$C(a,b,1)=\k \la x, y \ra/\la yx-ax^2-bxy-y^2 \ra$ or $C(a,b,0)=\k \la x, y \ra/\la yx-ax^2-bxy \ra$ for some $a,b\in \k$. Moreover, 
\begin{itemize}
\item[(1)] $C(a,b,0)$ is a graded twisted tensor product of $\k[x]$ and $\k[y]$ for any $a,b\in\k$;
\item[(2)] $C(a,b,1)$ is a graded twisted tensor product of $\k[x]$ and $\k[y]$ if and only if $a, b \in \k$ satisfy $f_n(a,b) \ne 0$ for all $n \geq 0$, where the $f_n(t,u)$ are a family of polynomials generalizing the family described in \cite[Section 6]{C-G}.
\end{itemize}
\end{thm}

In Proposition \ref{two dim as ttp}, we classify the algebras $C(a,b,0)$ and $C(a,b,1)$ up to isomorphism of graded twisted tensor products. We classify the algebras $C(a,b,1)$ and $C(a,b,0)$ up to isomorphism of graded algebras in Theorem \ref{two dim up to isom}. 

Section 4 is concerned with the classification of quadratic twisted tensor products of $\k[x,y]$ and $\k[z]$. The analysis is considerably more complicated than that of Section 3. We describe twisted tensor products $\k[x,y]\tsr_{\t} \k[z]$ in terms of a \emph{graded twisting map} $\t:\k[z]\tsr \k[x,y]\to \k[x,y]\tsr \k[z]$. By \cite[Theorem 1.2]{C-G}, a quadratic twisted tensor product of $\k[x,y]$ and $\k[z]$ is determined up to isomorphism by the values of $\t(z\tsr x)$ and $\t(z\tsr y)$. Suppressing the tensors, we write
\begin{align*}
\t(zx) &= ax^2+bxy+cy^2+dxz+eyz+fz^2,\tag{\textdagger} \label{tauDef} \\
\t(zy) &= Ax^2+Bxy+Cy^2+Dxz+Eyz+Fz^2,
\end{align*}
where the coefficients are elements of $\k$. Our main results in Section 4 are summarized in the following theorem.

\begin{thm}[Lemma \ref{JNFs}, Lemma \ref{cases}]
\label{introClassification3D}
A quadratic twisted tensor product of $\k[x,y]$ and $\k[z]$ is determined, up to isomorphism of twisted tensor products, by a twisting map $\t$ of the form {\rm (\ref{tauDef})}
where $D=F=0$ and $e,f,A\in\{0,1\}$. Moreover, $\t$ belongs to of one of the following types:
\begin{enumerate}
\item (Ore type) $f=0$,
\item (Reducible type) $f=1$, $A=0$, 
\item (Elliptic type) $f=1$, $A=1$, $d=-1$.
\end{enumerate}
\end{thm}

Not all combinations of the parameters produce twisted tensor products. For detailed descriptions of parameter restrictions in each case, see Theorems \ref{Ore extension classification}, \ref{ttps with G_2 = 0}, and \ref{characterization of elliptic type ttps} respectively. The reason for the ``reducible'' and ``elliptic'' terminology is due to the fact that, generically, a reducible-type algebra has a reducible point scheme and an elliptic-type algebra has an elliptic curve as its point scheme. (Here ``point scheme'' refers to the scheme that represents the functor of point modules, as in \cite{ATVI}.) We intend to study the non-commutative algebraic geometry of these algebras, and of twisted tensor products in general, in a forthcoming paper.

In \cite[Example 5.4]{C-G} we presented an example of Koszul algebras $A$ and $B$ such that $A\tsr_{\t} B$ is not Koszul. However, Koszul algebras are necessarily quadratic algebras, and, in the example we provided, the algebra $A\tsr_{\t} B$ is not quadratic. This motivated Problem \ref{quad example}. In Section 5 we consider the Koszul property for the algebras classified in Section 4 and present a family of examples that provide a negative answer to Problem \ref{quad example}. 

\begin{thm}[Theorem \ref{Koszul property for Ore extensions, A = 0}, Theorem \ref{Koszul property for T(g, d)}]
\label{introKoszul}
Let $T$ be a quadratic twisted tensor product of $\k[x,y]$ and $\k[z]$. If $T$ is of Ore type or reducible type, then $T$ is Koszul.
If $T$ is of elliptic type, then $T$ is Koszul if and only if $c-(a-1)(C+a-1)\neq 0$.
\end{thm}

The non-Koszul elliptic-type algebras provide examples of non-Koszul quadratic twisted tensor products of Koszul algebras, solving Problem \ref{quad example}. These counterexamples seem interesting in their own right, so we also study the Yoneda algebra of the non-Koszul elliptic-type algebras. In particular, Theorem \ref{Yoneda presentations} gives an explicit finite presentation of these algebras, in terms of generators and relations.

The classification of the algebras appearing in Theorem \ref{introClassification2D} includes all skew polynomial algebras and the Jordan plane $\k\la x,y\ra/\la yx-xy-y^2\ra$. These are the two-dimensional \emph{Artin-Schelter regular} algebras. Having classified quadratic twisted tensor products of $\k[x,y]$ and $\k[z]$ in Section 4, we take up the question of Artin-Schelter regularity in Section 6, where we prove the following result.

\begin{thm}[Theorem \ref{AS-regular algebras}]
\label{introASreg}
Let $T$ denote a quadratic twisted tensor product of $\k[x,y]$ and $\k[z]$.
\begin{itemize}
\item[(1)] If $T$ is an algebra of Ore type, then $T$ is AS-regular if and only if $T\cong \k[x,y][z;\s,\d]$ where $\s \in \End(\k[x,y])$ is invertible. 
\item[(2)] If $T$ is an algebra of reducible type, then $T$ is AS-regular if and only if $E \ne 0$ and $a+d \ne 0$.
\item[(3)] Assume that $\text{char}\, \k \ne 2$. If $T$ is an algebra of elliptic type, then $T$ is AS-regular if and only if $c-(a-1)(C+a-1) \ne 0$.
\end{itemize}

\end{thm} 

Finally, we remark that, in general, the graded twisted tensor product $R\tsr_{\t}S$ defined by a twisting map $\t:S\tsr R\to R\tsr S$ and the graded twisted tensor product $S\tsr_{\t'}R$ defined by a twisting map $\t':R\tsr S\to S\tsr R$ are not the same, and need not be related. However, if $R=\k[x,y]$ and $S=\k[z]$, then $R\tsr_{\t} S\cong (S\tsr_{\t^{\rm op}} R)^{\rm op}$ where $\t^{\rm op} = \z\t \z$ is a graded twisting map with $\z:R\tsr S\to S\tsr R$ given by $\z(r\tsr s)=s\tsr r$. Thus, by considering opposite algebras, our results characterize all graded twisted tensor products of $\k[x,y]$ and $\k[z]$, in either order.

\section{Preliminaries}
\label{preliminaries}

Throughout the paper, let $\k$ denote an algebraically closed field. Tensor products taken with respect to $\k$ are denoted by $\tsr$. We write $\k^*$ for $\k - \{0\}$.

We work extensively in categories where the objects are graded $\k$-vector spaces. If $V$ and $W$ are $\N$-graded $\k$-vector spaces, then $V \tsr W$ is $\N$-graded by the K\"unneth formula $$(V \tsr W)_m = \bigoplus_{k+l = m} V_k \tsr W_l.$$ Whenever we refer to $V\tsr W$ as a graded space, we assume this K\"unneth grading.  

The term {\it graded algebra} refers to a unital, associative $\k$-algebra, $A = \oplus_{n \geq 0} A_n$, that is {\it connected} ($A_0 = \k$) and generated by finitely many homogeneous elements. These assumptions imply the graded algebras we consider are $\N$-graded and {\it locally finite} ($\dim A_n < \infty$ for all $n \geq 0$). Almost all of the graded algebras in this paper are generated in homogeneous degree $1$. If $A$ is a graded algebra, we denote the (graded) multiplication map $A\tsr A\to A$ by $\m_A$. The kernel of the canonical graded algebra homomorphism $A\to \k$ is the graded radical of $A$, denoted $A_+ = \bigoplus_{i > 0} A_i$. 

\subsection{Twisted tensor products and twisting maps}
Let $A$ and $B$ be graded algebras. A {\emph{graded twisted tensor product}} of $A$ and $B$ is a triple $(C, i_A, i_B)$ consisting of a graded algebra $C$ and injective homomorphisms of graded algebras $i_A: A \to C$ and $i_B: B \to C$ such that the graded $\k$-linear map $A \tsr B \to C$ given by $a \tsr b \mapsto i_A(a)i_B(b)$ is an isomorphism of $\k$-vector spaces. 

Suppose that $\T = (C, i_A, i_B)$ and $\T' = (C', i'_A, i'_B)$ are two twisted tensor products of $A$ and $B$. We say that $\T$ and $\T'$ are {\it isomorphic} if there exist graded algebra isomorphisms $\a: A \to A$, $\b: B \to B$ and $\g: C \to C$ such that $\g i_A = i'_A \a$ and $\g i_B = i'_B \b$. Note that this notion of isomorphism is stronger than that of algebra isomorphism. Indeed it is possible to have non-isomorphic triples $\T$ and $\T'$ where $C$ and $C'$ are isomorphic as graded algebras (see Proposition \ref{two dim as ttp} and Theorem \ref{two dim up to isom}).

As shown in \cite{Cap}, the study of twisted tensor products is greatly facilitated by the notion of a twisting map. We adopt the most general graded version for our work below.
By a {\it graded twisting map} we mean a graded $\k$-linear map $\t: B \tsr A \to A \tsr B$ such that $\t(1 \tsr a) = a \tsr 1$ and $\t(b \tsr 1) = 1 \tsr b$ and $$\t(\mu_B \tsr \mu_A) = (\m_A \tsr \m_B)(1 \tsr \t \tsr 1)(\t \tsr \t)(1 \tsr \t \tsr 1).$$  The condition that $\t$ is graded is simply that $\t((A \tsr B)_n) \subseteq (B \tsr A)_n$ for all $n \geq 0$ with respect to the K\"unneth grading. It is common in the literature to see the phrase ``graded twisting map'' refer to the more restrictive condition: $\t(B_i \tsr A_j) \subseteq A_j \tsr B_i$ for all $i, j \geq 0$. Since a graded twisting map $\t$ satisfies $$\t(B_+ \tsr A_+) \subseteq (A_+ \tsr \k) \oplus (A_+ \tsr B_+) \oplus (\k \tsr B_+),$$ we call $\t$ \emph{one-sided} if either $$\t(B_+ \tsr A_+) \subseteq (A_+ \tsr \k) \oplus (A_+ \tsr B_+)$$ or $$\t(B_+ \tsr A_+) \subseteq (A_+ \tsr B_+) \oplus (\k \tsr B_+).$$ 

The relationship between twisting maps and twisted tensor products was established in the ungraded case in \cite{Cap} and in the graded case in \cite{C-G}.

\begin{prop}
\label{twisting maps and ttps}
\cite[Proposition 2.3]{C-G}
Let $A$ and $B$ be graded algebras. Let $\t: B \tsr A \to A \tsr B$ be a graded $\k$-linear map. Define  $\m_{\t}: A \tsr B \tsr A \tsr B \to A \tsr B$ by $\m_{\t} = (\m_A \tsr \m_B)(1 \tsr \t \tsr 1)$. Then $\t: B \tsr A \to A \tsr B$ is a graded twisting map if and only if $\m_{\t}$ defines an associative multiplication giving $A \tsr B$ the structure of a graded algebra.
\end{prop}

If $\t: B \tsr A \to A \tsr B$ is a graded twisting map, we use the notation $A \tsr_{\t} B$ for the algebra $(A \tsr B, \m_{\t})$. Note that the triple $(A \tsr_{\t} B, i_A, i_B)$, where $i_A: A \to A \tsr B$ and $i_B: B \to A \tsr B$ are the canonical inclusions, is a twisted tensor product of $A$ and $B$. We will abuse notation and also write $A \tsr_{\t} B$ for this triple.

The ungraded version of the following fundamental result first appeared in \cite[Proposition 2.7]{Cap}.

\begin{prop}
\label{ttp yields twisting map}
\cite[Proposition 2.4]{C-G}
Let $(C, i_A, i_B)$ be a graded twisted tensor product of graded algebras $A$ and $B$. Then there exists a unique graded twisting map $\t$ such that $(C, i_A, i_B)$ is isomorphic to $A \tsr_{\t} B$ as graded twisted tensor products of $A$ and $B$.
\end{prop}

The graded twisted tensor product $A \tsr_{\t} B$ can also be identified with a certain quotient of the free product algebra $A \ast B$. As a $\k$-vector space $$A \ast B = \bigoplus_{i \geq 0; \ \e_1, \e_2 \in \{0, 1\}} A^{\e_1}_+ \tsr (B_+ \tsr A_+)^{\tsr i} \tsr B_+^{\e_2}.$$ The algebra $A \ast B$ is $\N$-graded by the usual K\"unneth grading. Moreover, there are natural inclusions of $A$ and $B$ into $A \ast B$. Define an ideal of $A \ast B$ by $$I_{\t} = \la b \tsr a - \t(a \tsr b) : a \in A, b \in B \ra.$$ By \cite[Proposition 2.5]{C-G}, we have $A \tsr_{\t} B \cong (A \ast B)/I_{\t}$ as graded algebras.

\subsection{Quadratic twisted tensor products}
Let $C$ be a graded algebra that is generated in degree $1$, and let $T(C_1)$ denote the tensor algebra on the vector space $C_1$. The algebra $C$ is called {\it quadratic} if the kernel of the canonical projection $T(C_1) \to C$ is generated in degree 2. Let $\t: B \tsr A \to A \tsr B$ be a graded twisting map. In \cite{C-G}, the authors characterized when the algebra $A \tsr_{\t} B$ is quadratic in terms of the structure of $\t$. We briefly recall the relevant definitions to state this result.

If $V$ and $W$ are graded vector spaces and $f: V \to W$ is a graded linear map, we denote the degree-$n$ component of $f$ by $f_n$ and define $f_{\leq n} = \oplus_{i = 0}^n f_i$ and $f_{>n} = \oplus_{i > n} f_i$.

We say a graded linear map $t: (B \tsr A)_{\leq n} \to (A \tsr B)_{\leq n}$ is {\emph{graded twisting in degree}} $n$ if $t(1 \tsr a) = a \tsr 1$ and $t(b \tsr 1) = 1 \tsr b$ for all $a \in A_n$ and $b \in B_n$, and $$t_n(\m_B \tsr \m_A) = (\m_A \tsr \m_B)(1 \tsr t_{\leq n} \tsr 1)(t_{\leq n} \tsr t_{\leq n})(1 \tsr t_{\leq n} \tsr 1)$$ as maps defined on $(B \tsr B \tsr A \tsr A)_n$. If $t$ is graded twisting in degree $i$ for all $i \leq n$, we say $t$ is {\emph{graded twisting to degree $n$}}.

\begin{defn}\cite[Definition 4.2]{C-G}
\label{UEP defn}
A graded twisting map $\t: B \tsr A \to A \tsr B$ has the {\emph{unique extension property to degree $n$}} if, whenever $\t': B \tsr A \to A \tsr B$ is a graded linear map that is twisting to degree $n$ such that $\t_i = \t'_i$ for all $i < n$, it follows that $\t_n = \t'_n$.

The graded twisting map $\t$ has the {\emph{unique extension property}} if $\t$ has the unique extension property to degree $n$ for all $n \geq 3$.
\end{defn}

We note that every one-sided graded twisting map has the unique extension property (\cite[Proposition 5.2]{C-G}).


\begin{thm}
\label{quadratic iff uep}
Let $A$ and $B$ be quadratic algebras and let $\t: B \tsr A \to A \tsr B$ be a graded twisting map. The following are equivalent:
\begin{enumerate}
\item the graded algebra $A \tsr_{\t} B$ is quadratic,
\item the graded twisting map $\t$ has the unique extension property,
\item the ideal $I_{\t}$ of $A\ast B$ is generated in degree 2.
\end{enumerate} 
\end{thm}

\begin{proof}
The equivalence of (1) and (2) is \cite[Theorem 1.2]{C-G}. Statements (1) and (3) are equivalent by \cite[Proposition 2.5]{C-G}.
\end{proof}

When the algebra $A \tsr_{\t} B$ is quadratic we will refer to it as a \emph{quadratic twisted tensor product} of $A$ and $B$. 

The following result is very helpful in classifying twisted tensor products up to isomorphism.

\begin{prop}\cite[Proposition 2.2]{C-G}
\label{isomorphisms of ttps}
Let $A$ and $B$ be algebras, and let $\a: A \to A$ and $\b: B \to B$ be algebra automorphisms. If $\t: B \tsr A \to A \tsr B$ is a twisting map, then the map $\t': B \tsr A \to A \tsr B$ defined by $\t' = (\a \tsr \b) \t (\b^{-1} \tsr \a^{-1})$ is a twisting map. Furthermore, $A \tsr_{\t} B$ and $A \tsr_{\t'} B$ are isomorphic as twisted tensor products of $A$ and $B$.
\end{prop}


\subsection{Koszul algebras}
There are many equivalent ways to define the notion of a Koszul algebra; see the book by Polishchuk and Positselski \cite{PP}. 
In this paper, we call a graded algebra $A$ a \emph{Koszul algebra} if the trivial module $_A \k = A_0 = A/A_{+}$ admits a graded projective resolution $$\cdots \rightarrow P_3 \rightarrow P_2 \rightarrow P_1 \rightarrow P_0 \rightarrow \k \to 0,$$ where, for all $i \geq 0$, $P_i$ is generated in degree $i$. It is well known (see \cite{PP}, for example) that every Koszul algebra is quadratic.

\section{Quadratic twisted tensor products of $\k[x]$ and $\k[z]$}
\label{ttps of two one-variable polynomial algebras}

In this section we classify all of the quadratic twisted tensor products of two polynomial rings of one variable. This completes the work begun in \cite[Section 6]{C-G}; see especially Theorems 6.2, 6.5 and Proposition 6.6 of that paper. We also use this classification in the next section.

Throughout this section we fix the following notation. Let $A = \k[x]$, $B = \k[z]$, and let $\t:B\tsr A\to A\tsr B$ be a graded twisting map. Then $$\t(z \tsr x) = ax^2 \tsr 1 + bx \tsr z + 1 \tsr cz^2$$ for some $a, b, c \in \k$.  Let $$C = \k \la x, z \ra/ \la zx-ax^2-bxz-cz^2\ra.$$ To indicate dependence on the parameters we will also write $C = C(a, b, c)$. 

\begin{prop}\label{oneSidedTwoGen}
If $a = 0$ or $c = 0$, then $\t$ has the unique extension property, hence $A\tsr_{\t} B$ is a quadratic algebra isomorphic to $C(a,b,c)$.
\end{prop} 

\begin{proof}
The hypotheses imply the graded twisting map $\t$ is one-sided. The result follows from \cite[Proposition 5.2]{C-G} and Theorem \ref{quadratic iff uep}.
\end{proof}


When $ac \ne 0$, applying Proposition \ref{isomorphisms of ttps} to the automorphism of $B$ given by $z\mapsto z/c$ shows there is no loss of generality in assuming $c = 1$. Thus we consider $$C(a, b,1) = \k \la x, z \ra/\la zx-ax^2-bxz-z^2 \ra.$$ By Theorem \ref{quadratic iff uep}, determining when $\t$ has the unique extension property is equivalent to determining when $C$ is a graded twisted tensor product of $A$ and $B$. Note that for the latter to hold, it is necessary that the Hilbert series of $C$ be $(1-t)^{-2}$. 

\begin{lemma}
\label{quad hilb series}
The Hilbert series of the algebra $C(a, b,1)$ is $(1-t)^{-2}$ unless $a=1$ and $b=-1$.
\end{lemma}

\begin{proof}
By \cite[p. 126]{PP}, the Hilbert series of $C=C(a,b,1)$ is either $(1-t)^{-2}$ or the ``Fibonacci series" $\sum_{n} F_n t^n$, where $F_0 = 1$, $F_1 = 2$ and $F_{n+2} = F_{n+1} + F_n$, for $n \geq 0$. We claim the Hilbert series is  $\sum_{n} F_n t^n$ if and only if $a = 1$ and $b = -1$. 

To see this, order the generators of $C$ as $x < z$, and order monomials using left-lexicographic order. By the Diamond Lemma (see, for example, \cite{Berg}), the element  $$G = (1+b)zxz + (a-1)zx^2+(b^2-a)x^2z+a(b+1)x^3$$ of the free algebra $\k\la x,z\ra$ is (up to scaling) the only degree-3 element of a Gr\"obner basis with respect to the chosen monomial term order.
 
It follows that $\dim C_3 = 5$ if and only if $G = 0$, which happens if and only if $a = 1$ and $b = -1$.
\end{proof}

Next we define some sequences of polynomials that are the key to the rest of the classification problem. Let  $\{e_n(t, u)\}, \{f_n(t, u)\}, \{g_n(t, u)\}, \{h_n(t, u)\}$ in $\k[t, u]$ be defined as follows:
\begin{align*}
e_0(t,u) &= 1, &f_0(t,u) &= 1, \\ 
e_n(t,u) &= ue_{n-1}(t,u)+f_{n-1}(t,u), &f_n(t,u) &= -t e_{n-1}(t,u)+f_{n-1}(t,u), & n&\ge 1,\\
g_n(t,u) &= (1-u)e_n(t,u)-f_n(t,u), &h_n(t,u) &= -te_n(t,u), &  n&\ge 0.
\end{align*} 

\begin{lemma}\label{key relation}\ 
\begin{enumerate}
\item For all $n\ge 0$, $f_n(t,-t )= (1-t)^n$.
\item For all $n\ge 0$, the relation $$e_nzx^nz = f_nzx^{n+1}+g_nx^{n+1}z+h_nx^{n+2}$$ holds in $C(a,b,1)$, where $e_n = e_n(a, b)$, $f_n = f_n(a, b)$, $g_n = g_n(a,b)$, and $h_n = h_n(a,b)$.
\item If $f_i(a,b)\neq 0$ for all $1\le i\le n-1$, then  $zx^k\in {\rm Span}\{x^i z^j : i, j \geq 0\}\subseteq C(a,b,1)$ for $1 \leq k \leq n$. 
\end{enumerate}
\end{lemma}

\begin{proof}
From the recursive formulas it is clear that $e_n(t, -t) = f_n(t,-t)$ for all $n\ge 0$. Thus for $n\ge 1$ we have
\begin{align*}
f_n(t,-t) &= -te_{n-1}(t,-t)+f_{n-1}(t,-t) \\
&= -tf_{n-1}(t,-t)+f_{n-1}(t,-t) \\
&= f_{n-1}(t,-t)(1-t),
\end{align*}
and hence (1) follows by induction.

The relation in (2) holds for $n = 0$ by definition of $C(a,b,1)$. Suppose that the relation  
\[e_{n-1}zx^{n-1}z = f_{n-1}zx^{n}+g_{n-1}x^{n}z+h_{n-1}x^{n+1} \tag{\textasteriskcentered} \label{inductHyp}\]
holds for some $n \geq 1$.
Multiplying through by $z$ on the right and using the defining relation of $C(a,b,1)$ yields
$$e_{n-1}zx^{n-1}(zx-bxz-ax^2) = f_{n-1}zx^{n}z+g_{n-1}x^{n}(zx-bxz-ax^2)+h_{n-1}x^{n+1}z.$$
Next, using the relation (\ref{inductHyp})  on the first term, rearranging and collecting terms we have 
$$(be_{n-1}+f_{n-1})zx^nz = (-ae_{n-1}+f_{n-1})zx^{n+1}+(bg_{n-1}-h_{n-1})x^{n+1}z+(h_{n-1}+ag_{n-1})x^{n+2}.$$
By the recursive definitions
\begin{align*}
g_n &= (1-b)[be_{n-1}+f_{n-1}]-[-ae_{n-1}+f_{n-1}] \\
&= b[(1-b)e_{n-1}-f_{n-1}]+ae_{n-1} \\
&= bg_{n-1}-h_{n-1}, \\
h_n &= -ae_n \\
&= -a(be_{n-1}+f_{n-1}) \\
&= -ae_{n-1}+a[(1-b)e_{n-1}-f_{n-1}] \\
&= h_{n-1}+ag_{n-1},
\end{align*}
therefore we see that we have 
$$e_nzx^nz = f_nzx^{n+1}+g_nx^{n+1}z+h_nx^{n+2},$$ as desired.

Statement (3) now follows by a straightforward induction.

\end{proof}

\begin{thm} 
\label{quadratic ttps of polynomial rings of one variable}
The algebra $C(a,b,1)$ is a graded twisted tensor product of $\k[x]$ and $\k[z]$ if and only if $a, b \in \k$ satisfy $f_n(a,b) \ne 0$ for all $n \geq 0$. 
\end{thm}

\begin{proof}
For all $n \geq 0$, write $e_n, f_n, g_n, h_n$ for $e_n(a,b), f_n(a,b), g_n(a,b), h_n(a,b)$, respectively. Let $A=\k[x]$, $B=\k[z]$, and $C=C(a,b,1)$.

If $a=0$, then $f_n=1$ for all $n\ge 0$, and by Proposition \ref{oneSidedTwoGen}, $C$ is a graded twisted tensor product of $A$ and $B$. Henceforth we assume $a\neq 0$.

Suppose $f_n \ne 0$ for all $n \geq 0$.
Let $i_A:A\to C$ and $i_B:B\to C$ be the graded algebra homomorphisms determined respectively by $i_A(x)=x$ and $i_B(z)=z$. Let  $$S={\rm Span}\{x^i z^j : i, j \geq 0\}\subseteq C.$$ By Lemma \ref{key relation}(3), since $f_n \ne 0$ for all $n \geq 0$, we know $zx^k \in S$ for all $k \geq 0$. It follows that $C = S$. In particular, the canonical linear map $$(i_A,i_B):A \tsr B \to C$$ given by $a\tsr b\mapsto i_A(a)i_B(b)$ is surjective.

Since $f_1 = 1-a \ne 0$, we know $a \ne 1$. Thus by Lemma \ref{quad hilb series}, the Hilbert series of $C$ is $(1-t)^{-2}$. Hence $i_A$ and $i_B$ are injective and $(i_A,i_B):A \tsr B \to C$ is a linear isomorphism. It follows that $C$ is a twisted tensor product of $A$ and $B$.

Conversely, suppose that $a, b \in \k$ are such that $f_n(a,b) = 0$ for some $n \geq 1$.

If $f_1 = 0$, then $a = 1$. Suppose $b = -1$. Then by Lemma \ref{quad hilb series} the Hilbert series of $C$ is not $(1-t)^{-2}$,  so $C$ is not a twisted tensor product of $A$ and $B$. If $b \ne -1$, the relation in Lemma \ref{key relation}(2) yields $$(1+b)zxz +(b^2-1)x^2z+(b+1)x^3 = 0.$$
Substituting $zx = ax^2+bxz+z^2$ results in $$(1+b)(ax^2+bxz+z^2)z + (b^2-1)x^2z+(b+1)x^3 = 0.$$ Considering that the coefficient of $x^3$ is nonzero we see that this is a nontrivial dependence relation in $S$. Thus $C$ is not a twisted tensor product of $A$ and $B$. This concludes the case $f_1 = 0$.

Henceforth we assume $n\ge 2$ is minimal such that $f_n = 0$. Since $f_1\neq 0$, we have $a\neq 1$.  Thus Lemma \ref{key relation}(1) implies $b \ne -a$. We claim that $e_n \ne 0$. 

Suppose, to the contrary, that $e_n = 0$. Then the recurrence formulas for $f_n$ and $e_n$ give
$f_{n-1} = ae_{n-1}$ and $f_{n-1} = -be_{n-1}.$ Since $b \ne -a$, we have $e_{n-1} = 0$ and hence $f_{n-1}=0$,  contradicting the minimality of $n$. Thus $e_n\ne 0$.

By Lemma \ref{key relation}(2) the relation 
\[e_nzx^nz - g_nx^{n+1}z + ae_nx^{n+2} = 0\tag{\textasteriskcentered\textasteriskcentered} \label{depRel}\] holds in $C$, and since $f_1, \ldots, f_{n-1}$ are all nonzero we know that $zx^n \in S$ by Lemma \ref{key relation}(3). Since $a\neq 0$ and $e_n\neq 0$, the coefficient of $x^{n+2}$ in (\ref{depRel}) is nonzero, so (\ref{depRel}) is a nontrivial dependence relation in $S$. Thus $C$ is not a twisted tensor product of $A$ and $B$.

\end{proof}

Noting that under the change of variables $z\mapsto z/c$ employed above, the coefficient of $x^2$ in $\t(z\tsr x)$ becomes $ac$, we have the following.

\begin{cor}
\label{abc}
Let $A = \k[x]$, $B = \k[z]$, and let $\t:B\tsr A\to A\tsr B$ be a graded twisting map. 
Let $a, b, c \in \k$ be given by $$\t(z \tsr x) = ax^2 \tsr 1 + bx \tsr z + 1 \tsr cz^2.$$
Then $A\tsr_{\t} B$ is quadratic unless $f_n(ac,b)=0$ for some $n\ge 0$.
\end{cor}

To end this section we will classify the quadratic graded twisted tensor products of $A=\k[x]$ and $B=\k[z]$ up to: (1) isomorphism of twisted tensor products of $A$ and $B$, and (2) isomorphism of graded algebras. 


\subsection{Classification up to isomorphism of graded twisted tensor products}

Recall that $$C(a, b, c) = \k \la x, z \ra/\la zx-ax^2-bxz-cz^2 \ra.$$ By Corollary \ref{abc}, the algebra $C(a, b, c)$ is a twisted tensor product of $A$ and $B$ with respect to the obvious inclusions of $A$ and $B$ if and only if $f_n(ac,b)\neq 0$ for all $n\ge 0$. In the classification that follows, we assume that $f_n(ac,b)\neq 0$ for all $n\ge 0$. Above we showed that if $c\ne 0$, then $C(a,b,c)\cong C(ac,b,1)$ as graded twisted tensor products. Similarly, if $c=0\neq a$, the automorphism $x\mapsto x/a$ induces an isomorphism $C(a,b,0)\cong C(1,b,0)$. Otherwise, as twisted tensor products of $A$ and $B$, the $C(a, b, c)$ are completely rigid in the following sense.

\begin{prop}
\label{two dim as ttp}
If $C(a, b,1) \cong C(a', b',1)$ as twisted tensor products of $A$ and $B$, then $a = a'$ and $b = b'$. If $C(1,b,0)\cong C(1,b',0)$ or $C(0,b,0)\cong C(0,b',0)$ as graded twisted tensor products of $A$ and $B$, then $b=b'$.
\end{prop}

\begin{proof}
By definition, any isomorphism of twisted tensor products $\g: C(a, b, 1) \to C(a', b', 1)$ must commute with the inclusions. It follows that $\g(x) = \a x$ and $\g(z) = \b z$ for some $\a, \b \in \k^*$. Considering the defining relations of $C(a, b, 1)$ and $C(a', b', 1)$ makes it evident that $a = a'$ and $b = b'$. An analogous argument holds in the $c=0$ case.
\end{proof}

\subsection{Classification up to isomorphism of graded algebras}
By Proposition \ref{two dim as ttp}, it suffices to classify the twisted tensor products: $C(a,b,1)$, $C(1,b,0)$, and $C(0,b,0)$ up to isomorphism.

Let $R = \k \la x, z \ra/ \la f \ra$, where $0 \ne f$ is homogeneous of degree 2. It is well known  (see, for example, \cite[Exercise 2.4.3 (2)]{Rogalski}) that $R$ is isomorphic (as a graded algebra) to exactly one of: 
\begin{itemize}
\item[(i)] $\k_q[x, z] = \k \la x, z \ra/ \la zx-qxz \ra$, for some $q \in \k$; 
\item[(ii)] $J(x, z) = \k \la x, z \ra/\la zx-xz-z^2\ra$, (the {\it Jordan plane}); 
\item[(iii)] $\k \la x, z \ra/\la x^2 \ra$. 
\end{itemize}
We handle the $c=0$ case first.

\begin{prop}\ The following graded algebra isomorphisms hold.
\begin{enumerate}
\item $C(1,b,0)\cong \k_b[x,z]$ if $b\neq 1$;
\item $C(0,b,0)=\k_b[x,z]$ if $b\neq 0$; 
\item $C(1,1,0)= J(x,z)$;
\item $C(1,0,0)\cong C(0,0,0)=\k \la x, z \ra/\la zx \ra$.
\end{enumerate}
\end{prop}

\begin{proof}
The first isomorphism can be obtained by $x\mapsto (1-b)x$ and $z\mapsto x+z$. The rest are obvious.
\end{proof}

To determine the graded algebra isomorphism type of $C(a, b,1)$ we use the following result.

\begin{lemma}\cite[Exercise 2.4.3]{Rogalski}
\label{rogalski lemma}
Let $R = \k \la x, z \ra/\la f \ra,$ where $0 \ne f$ is homogeneous of degree 2. Write $f = x \phi(x) + z \phi(z),$ for some unique linear transformation $\phi: \k x \oplus \k z \to  \k x \oplus \k z$. If $\phi(x) = m_{11} x + m_{12}z$ and $\phi(z) = m_{21} x + m_{22} z$, define a matrix $M(\phi) = (m_{ij})$. Then $R(\phi) \cong R(\phi')$ as graded algebras if and only if there exists an invertible matrix $N$ such that $M' = N^t M N$.
\end{lemma}

\begin{thm}
\label{two dim up to isom}
Suppose that $C(a, b,1)$ is a quadratic twisted tensor product of $A$ and $B$. 
Then, as graded algebras,
\begin{itemize}
\item[(1)] $C(a, -1,1) \cong \k_{-1}[x,z]$, if ${\rm char}\ \k\neq 2$;
\item[(2)] if $4a-(b-1)^2 = 0$, then $C(a, b,1) \cong J(x, z)$;
\item[(3)] if $b \ne -1$ and $4a-(b-1)^2 \ne 0$, then $C(a, b,1) \cong \k_q[x,z]$, where $q$ satisfies $(a+b)q^2+(2a-b^2-1)q + a + b = 0$.
\end{itemize}
\end{thm}

Note that if ${\rm char}\ \k\neq 2$ and $b=-1$ then $4a-(b-1)^2=0$ implies $a=1$. This violates  Theorem \ref{quadratic ttps of polynomial rings of one variable}, so there is no conflict between statements (1) and (2).

\begin{proof}
First, note that by Theorem \ref{quadratic ttps of polynomial rings of one variable}, $f_1(a,b)\neq 0$ so $a\neq 1$.
To prove the theorem, we simply write down an invertible matrix $N$ and check the condition of Lemma \ref{rogalski lemma}. 

Using the notation of Lemma \ref{rogalski lemma}, for $\k_q[x,z]$, $J(x,z)$, and $C(a,b,1)$ we have $$M_q = \begin{bmatrix} 0 & -q \\ 1 & 0 \end{bmatrix},\qquad M_J = \begin{bmatrix} 0 & -1 \\ 1 & -1 \end{bmatrix},\qquad\text{ and }\qquad M_C = \begin{bmatrix} -a & -b \\1 & -1 \end{bmatrix}$$ respectively.

To prove (1), let $$N = \begin{bmatrix} \frac{1+\sqrt{1-a}}{2} & -\frac{1}{2} \\ -1+\sqrt{1-a} & 1 \end{bmatrix}.$$ The determinant of $N$ is $\sqrt{1-a}$, which is nonzero since $a \ne 1$, so $N$ is invertible. It is straightforward to check that $N^t M_{-1} N = M_C$.

For (2), assume that $4a-(b-1)^2 = 0$.  Let $\sqrt{a}$ be chosen such that $2\sqrt{a}=b-1$ and put $$N = \begin{bmatrix} 1+\sqrt{a} & 0 \\ \sqrt{a} & 1\end{bmatrix}.$$ Clearly $N$ is invertible and one checks that $N^t M_J N = M_C$.

Finally, for (3) assume that $b \ne -1$ and $4a-(b-1)^2 \ne 0$. Suppose that $q \in \k$ satisfies $(a+b)q^2+(2a-b^2-1)q + a + b = 0$. We claim that $q \in \k - \{1, -1\}$. To see this suppose $q \in \{1, -1\}$. If $q = 1$, then $4a-(b-1)^2 = 0$; if $q = -1$, then $b = -1$. We conclude that $q \in \k - \{1, -1\}$. Define $$N = \begin{bmatrix} \frac{bq-1}{q^2-1} & \frac{1}{q-1} \\ \frac{b-q}{q+1} & 1 \end{bmatrix}.$$ Then $\det N = (1+b)/(1+q) \ne 0$, so $N$ is invertible and one checks that $N^t M_q N = M_C$.

\end{proof}

\section{Quadratic twisted tensor products of $\k[x,y]$ and $\k[z]$}
\label{sec:QuadTTPs}

Throughout this section, and for the rest of the paper, we fix the notation $R = \k[x, y]$ and $S = \k[z]$. For readibility, henceforth, when we write $T$ is a (quadratic) graded twisted tensor product, we mean $T$ is a (quadratic) graded twisted tensor product of $R$ and $S$. Also, if $T$ and $T'$ are graded twisted tensor products and we write $T$ is isomorphic to $T'$, then we mean that $T$ and $T'$ are isomorphic as graded twisted tensor products of $R$ and $S$. The main result of this section is the determination of all of the isomorphism classes of the quadratic twisted tensor products of $R$ and $S$. Let $a, b, c, d, e, f, A, B, C, D, E, F \in \k$ and define elements of the tensor algebra $\k \la x, y, z \ra$ (suppressing tensors)
\begin{align*}
\t(zx) &= ax^2+bxy+cy^2+dxz+eyz+fz^2, \\
\t(zy) &= Ax^2+Bxy+Cy^2+Dxz+Eyz+Fz^2.
\end{align*}

The problem under consideration is to determine when the algebra $$T = \k \la x, y, z \ra/ \la zx-\t(zx), zy-\t(zy), xy-yx \ra$$ is a graded twisted tensor product of $R$ and $S$ relative to the obvious maps $i_R:R \to T$ and $i_S:S \to T$. To indicate the dependence of $T$ on the parameters, we will also use the notation 
$$T = T(a, b, c, d, e, f; A, B, C, D, E, F).$$
We begin with a reduction in the number of parameters.

\begin{lemma}
\label{JNFs}
If $T$ is a graded twisted tensor product, then $T$ is isomorphic to $T(a', b', c', d', e', f'; A', B', C', 0, E', 0)$ where $e', f'\in \{0,1\}$. Furthermore, if $e'=0$ then $A'\in\{0,1\}$, if $e'=A'=0$, then $C'\in\{0,1\},$ and if $e'=1$, then $d'=E'$.
\end{lemma}

\begin{proof}
We will apply Proposition \ref{isomorphisms of ttps} repeatedly to automorphisms of $R$ and $S$.

If $F=0\neq f$, applying Proposition \ref{isomorphisms of ttps} to the automorphism of $S$ given by $z \mapsto f^{-1} z$, we see that $T\cong T(a', b', c', d', e', 1; A', B', C', D', E', 0).$

If $F\neq 0$, then applying Proposition \ref{isomorphisms of ttps} to the automorphism of $R$ given by $x \leftrightarrow y$ and the automorphism $z \mapsto F^{-1} z$, we have
$$T\cong T(a', b', c', d', e', 1; A', B', C', D', E', F').$$
Then we use the automorphism of $R$ given by $x \mapsto x$, $y \mapsto y + F'x$ to see that, without loss of generality, $F' = 0$ and $f'=1$. 

We have proved that if $T$ is any twisted tensor product of $R$ and $S$, then 
$$T\cong T(a', b', c', d', e', f'; A', B', C', D', E', 0)$$
for $f'\in \{0,1\}$, so we assume henceforth that $F=0$ and $f\in \{0,1\}$.

The graded automorphism group of $R$ is ${\rm GL}_2(\k)$, where an invertible matrix $P=\begin{bmatrix} p_{11} & p_{12}\\ p_{21} & p_{22}\\ \end{bmatrix}$ determines a graded algebra automorphism $\a_P:R\to R$ by $\a_P(x) = p_{11}x+p_{12}y$ and $\a_P(y)=p_{21}x+p_{22}y$. Since $\k$ is algebraically closed we choose an invertible matrix $P$ such that $P^{-1} \begin{bmatrix} d & e \\D & E \end{bmatrix} P$ is in Jordan normal form. Applying Proposition \ref{isomorphisms of ttps} to the automorphism $\a_P$, we have
$$T\cong T(a', b', c', d', e', f; A', B', C', 0, E', 0)$$
where $e'=0$, or $e'=1$ and $d'=E'$. Note that this isomorphism does not affect the coefficients of $z^2$. 

If $e'=0$ and $A'\neq 0$, we may rescale $y\mapsto A'y$ to obtain
$$T\cong T(a'', b'', c'', d'', 0, f; 1, B'', C'', 0, E'', 0).$$
To complete the proof, if $e'=A'=0$, and $C'\neq 0$, we may rescale $y\mapsto C'^{-1}y$ to obtain
$$T\cong T(a'', b'', c'', d'', 0, f; 0, B'', 1, 0, E'', 0).$$
\end{proof}

We refer to the isomorphism class of Lemma \ref{JNFs} as the \emph{Jordan normal form} of the twisted tensor product $T$. 




Let us first dispense with the case $f=0$. Let $\s\in \End(R)$ be a graded ring endomorphism and let $\d$ be a graded left $\s$-derivation, meaning $\d(r_1r_2)=\s(r_1)\d(r_2)+\d(r_1)r_2$ for all $r_1, r_2\in R$. Then one may construct the \emph{graded (left) Ore extension} $R[z;\s,\d]$, which is freely generated as an $R$-algebra by $z$, subject to the relations $zr=\s(r)z+\d(r)$ for $r\in R$. 


\begin{prop}
\label{oneSidedOre}
The algebra $T$ is a graded twisted tensor product with $f=F=0$ if and only if $T$ is a graded Ore extension.
\end{prop}

\begin{proof}
Suppose $\s$ is a graded ring endomorphism of $R$ and $\d$ is a graded left $\s$-derivation. From the relations of $R[z;\s,\d]$ it is clear that the set $\S = \{x^i y^j z^k : i, j, k  \geq 0\}$ is a $\k$-linear spanning set. Since $z$ is a free generator and $\{x^iy^j : i, j\ge 0\}$ is a $\k$-basis for $R$, it follows that $\S$ is a $\k$-basis for $R[z;\s,\d]$. Hence the $\k$-linear map $R\tsr S\to R[z;\s,\d]$ given by $r\tsr z^i\mapsto rz^i$ is an isomorphism, and $R[z;\s,\d]$ is a graded twisted tensor product. The graded twisting map associated to $R[z;\s,\d]$ is given by $\t(z\tsr r)=\s(r)\tsr z+\d(r) \tsr 1$, and it follows that $f=F=0$.

Now suppose $T$ is a graded twisted tensor product of $R$ and $S$ with $f=F=0$. Since $T$ is quadratic, the definitions of $\t(z\tsr x)$ and $\t(z\tsr y)$ uniquely determine a one-sided graded twisting map $\t$ for $T$ by Theorem \ref{quadratic iff uep}. Since $\t$ respects multiplication in $R$, we have $\t(z\tsr r)\in (R\tsr \k z) \oplus (R\tsr \k 1_S)$ for all $r\in R$. Define graded $\k$-linear maps $\s, \d:R\to R$ by $\t(z\tsr r) = \s(r)\tsr z + \d(r)\tsr 1_S$. Then for any $r_1, r_2\in R$, we have
\begin{align*}
\s(r_1r_2)\tsr z+\d(r_1r_2)\tsr 1_S&=\t(z\tsr r_1r_2)\\
&=  (\m_R\tsr 1)(1\tsr \t)((\s(r_1)\tsr z + \d(r_1)\tsr 1_S)\tsr r_2)\\
&=\s(r_1)\s(r_2)\tsr z + (\s(r_1)\d(r_2)+\d(r_1)r_2)\tsr 1_S.
\end{align*}
This shows $\s$ is a ring endomorphism and $\d$ is a (left) $\s$-derivation. By Theorem \ref{quadratic iff uep}(3), for all $r\in R$, the relations $zi_R(r)=\s(i_R(r))z+\d(i_R(r))$ hold in $T$. Since $T$ is a graded twisted tensor product, $\S = \{x^i y^j z^k : i, j, k  \geq 0\}$ is a $\k$-basis for $T$, thus $T$ is freely generated as an $i_R(R)$-algebra by $z$. Thus $T$ is a graded Ore extension.
\end{proof}

When $T$ is isomorphic to a graded twisted tensor product with $f=F=0$, we say $T$ is of \emph{Ore type}. 

\subsection{The Ore-type case}

Let $\s \in \End(R)$ be the ring endomorphism given by $\s(x) = dx + ey$ and $\s(y) = Dx+Ey$. Proposition \ref{oneSidedOre} shows that to determine which values of the parameters ensure $T$ is a graded twisted tensor product with $f=F=0$, it suffices to characterize when the formulas $$\d(x) = ax^2+bxy+cy^2, \ \ \ \d(y) = Ax^2+Bxy+Cy^2$$ determine a well-defined left $\s$-derivation $\d:R\to R$. In turn, the definitions in the last display determine a well-defined left $\s$-derivation $\d:R\to R$ if and only if $\d(xy)=\d(yx)$, or equivalently if $$\s(x)\d(y)+\d(x)y = \s(y)\d(x) + \d(y) x, $$ holds in the ring $R$.


\begin{thm}
\label{Ore extension classification}
Suppose $T$ is an Ore-type twisted tensor product in Jordan normal form. 
\begin{itemize}
\item[(1)] If $e=0$, then $A\in\{0,1\}$, and if $A=0$, then $C\in\{0,1\}$. Furthermore, one of the following mutually exclusive cases is true:
\begin{itemize}
\item[(i)] $d =E = 1$ and $a, b, c, B$ are arbitrary;
\item[(ii)] $d \ne 1$, $E = 1$, $A = B = C = 0$, and $a, b, c$ are arbitrary;
\item[(iii)] $d =1$, $E \ne 1$, $a = b = c = 0$, and $B$ is arbitrary;
\item[(iv)] $d \ne 1$, $E \ne 1$, $A = c = 0$, $a=B(d-1)/(E-1)$, $b=C(d-1)/(E-1)$, and $B$ is arbitrary.
\end{itemize}

\item[(2)] If $e=1$, then $d=E$, and one of the following mutually exclusive cases is true:

\begin{itemize}
\item[(i)] $d = 1$, $A = B = C = 0$, and $a, b, c$ are arbitrary;
\item[(ii)] $d \ne 1$, $A = 0$, $B=a=(b-C)(d-1)$, $c=C/(d-1)$, and $b$ is arbitrary.
\end{itemize}

\end{itemize}

\end{thm}

\begin{proof} 
Assume $T$ is an Ore-type twisted tensor product in Jordan normal form, so $f=F=D=0$ and $e\in\{0,1\}$.

For (1), suppose that $e = 0$. By Lemma \ref{JNFs}, we have $A\in\{0,1\}$ and if $A=0$ then $C\in\{0,1\}$.  Equating coefficients in $\s(x)\d(y)+\d(x)y = \s(y)\d(x) + \d(y) x$ yields the system of equations
\begin{align*}
A(d-1) &= 0 \\
B(d-1)+a(1-E) &= 0 \\
C(d-1)+b(1-E) &= 0 \\
c(E-1) &= 0.
\end{align*}
The result then follows by considering the possible normal forms and whether $d$ or $E$ equals 1.

The proof of (2) is similar, and is left to the reader.
\end{proof}

%
%
%
%

This completes the classification, up to isomorphism, of twisted tensor products of Ore type.

We now turn our attention to the case where $T$ is a quadratic twisted tensor product in Jordan normal form with $f=1$. As the associated twisting map is no longer one-sided, this case is more difficult. Surprisingly, as we will show, knowing $\dim T_3=10$ is sufficient for proving $T$ is a quadratic twisted tensor product. For this reason, the remainder of this section relies upon a Gr\"obner-basis calculation. 

First we record a useful fact. 

%
%

\begin{prop}
\label{not a zero of fn}
If $T$ is a twisted tensor product in Jordan normal form with $A=0$, $E \ne 0$ and $f = 1$, then $y$ is normal in $T$ and $(a,d)$ is not a zero of $f_n(t,u)$ for all $n\ge 1$. In particular, $a\neq 1$.
\end{prop}

\begin{proof}
Let $T$ be a twisted tensor product in Jordan normal form with $A = 0$, $E \ne 0$ and $f=1$. Then $T$ may be presented as $$\k\la x, y, z \ra/\la xy-yx, zx-ax^2-byx-cy^2-dxz-eyz-z^2, zy-Byx-Cy^2-Eyz \ra.$$
Since $E\neq 0$, it is clear from this presentation that $y$ is normal in $T$. 

Next, it is easy to show that $$T/\la y \ra \cong \k \la x, z \ra/\la zx-ax^2-dxz-z^2 \ra.$$ 


By assumption, $T$ has $\{y^i x^j z^k : i, j, k \geq 0\}$ as a $\k$-basis, so it follows that $T/\la y \ra$ has $\{x^j z^k : j, k \geq 0\}$ as a $\k$-basis. Hence $T/\la y \ra$ is a twisted tensor product of $\k[x]$ and $\k[z]$. Then Theorem \ref{quadratic ttps of polynomial rings of one variable} implies that $(a, d)$ is not a zero of $f_n(t,u)$ for all $n \geq 1$.

Finally, recall that $f_1(t,u) = 1-t$, so we have $a \neq 1$.
\end{proof}

We briefly recall some terminology and notation for Gr\"obner-basis calculations, loosely following \cite{Berg}.

Let $\k\la X\ra$ be the free algebra on a (finite) set $X$. Fix a total order on $X$ and well-order monomials in $\k\la X\ra$ using degree-lexicographic order. Let $\mathscr B$ be a set of homogeneous elements of the form $W-f_W$ where $W$ is a monomial in $\k\la X\ra$  and $f_W\in\k\la X\ra$ such that $f$ is a linear combination of monomials less than $W$. Denote the set of all monomials $W$ such that $W-f_W\in\mathscr B$ for some $f_W\in\k\la X\ra$ by ${\rm ht}(\mathscr B)$; this is the set of \emph{high terms} of $\mathscr B$.
For simplicity, and to suit our purposes, we assume $\mathscr B$ has been chosen so that  $W\in {\rm ht}(\mathscr B)$ implies no subword of $W$ is in ${\rm ht}(\mathscr B)$.

If $m$ and $m'$ are monomials and $W-f_{W} \in\mathscr B$, let $r_{mWm'}:\k\la X\ra\to \k\la X\ra$ be the $\k$-linear map that is the identity on all monomials of $\k\la X\ra$ except $r_{mW m'}(mWm')=mf_Wm'$. We refer to the maps $r_{mWm'}$ as \emph{reductions}. Note that any infinite sequence of reductions applied to an element of $\k\la X\ra$ eventually stabilizes, though two such sequences need not result in the same element of $\k\la X\ra$. An \emph{overlap} occurs when there exist monomials $m, m', m''$ and elements $W_1-f_{W_1}, W_2-f_{W_2}\in \mathscr B$ such that $W_{1}=mm'$ and $W_{2}=m'm''$. An overlap is \emph{resolvable} if there are sequences of reductions $r$ and $r'$ such that $r(f_{W_1}m'')=r'(mf_{W_2})$. 

Consider the graded algebra $R=\k\la X\ra/\la\mathscr B\ra$. If all overlaps of $\mathscr B$ in homogeneous degrees $\le d$ are resolvable, then monomials of degree $\le d$ that do not contain any element of ${\rm ht}(\mathscr B)$ as a subword form a $\k$-basis for $R_{\le d}$. In this case we say $\mathscr B$ is a \emph{Gr\"obner basis to degree }$d$. If an overlap does not resolve, one may ``force'' it to resolve by adding to $\mathscr B$ a new element $W-f_W$ where $W$ is the largest monomial such that $\l(W-f_W)=r(f_{W_1}C-Af_{W_2})$ for some $\l\in\k^*$ and any sequence of reductions $r$ that eventually stabilizes on $f_{W_1}C-Af_{W_2}$. Note that adding such elements to $\mathscr B$ does not change $R$.

We apply this theory to the algebra $T$ in Jordan normal form. Order the monomials in $\k \la x, y, z \ra$ with left lexicographical order determined by $y < x < z$. Let 
$$\mathscr B=\{ xy-yx, z^2-(zx-ax^2-byx-cy^2-dxz-eyz), zy-(Ax^2+Byx+Cy^2+Eyz)\}.$$
Under the chosen term-order there are two overlaps in degree 3: $z^3$ and $z^2y$. Reducing the differences $(f_{z^2})z-z(f_{z^2})$ and  $(f_{z^2})y-z(f_{zy})$ as much as possible yields the following expressions:

\begin{align*}
G_1 &= (1+d) zxz +(a-1)zx^2 - (a-d^2-eA) x^2z+ (a+ad+bA)x^3 +(b+e)Eyzx\\
& +(eB-deE+2de-b)yxz+(b+bd+bB+cA+cAE+ae-aeE)yx^2\\
&+(cE^2-c+eC-e^2E+e^2)y^2z  +(c+cd+bC+cB+cBE-beE+be)y^2x\\
&+(cC+cCE-ceE+ce)y^3, \\
G_2 &= -Azx^2 - AEx^2z +(A-dA-AB)x^3+(E-BE-E^2)yzx\\
&-(dE+BE-dE^2)yxz+(B-a-dB-B^2-AC-ACE+aE^2-eA)yx^2\\
&-(CE^2+CE+eE-eE^2)y^2z+(C-b-dC-2BC-BCE+bE^2-eB)y^2x\\
 &-(c+C^2+C^2E-cE^2+eC)y^3.\\
\end{align*}

\begin{lemma}
\label{cases}
We have $\dim T_3=10$ if and only if $\dim {\rm Span}_{\k}\{G_1, G_2\}=1$.
If $T$ is a twisted tensor product in Jordan normal form with $f=1$, then either
\begin{enumerate}
\item $G_2= 0\neq G_1$ in $\k\la x,y,z\ra$ and $(a,d)$ is not a root of any $f_n(t,u)$, or
\item $G_2\neq 0$, $e=0$, $d=-1$, $A=1$ and $G_1=(1-a)G_2$.
\end{enumerate}
\end{lemma}

\begin{proof}
There are eleven monomials of degree 3 in $\k\la x,y,z\ra$ that do not contain an element of ${\rm ht}(\mathscr B)=\{xy, zy, z^2\}$ as a subword. If $G_1=G_2=0$, then $\mathscr B$ is a Gr\"obner basis for $T$ to degree 3 and $\dim T_3=11$. If $G_1$ and $G_2$ are linearly independent, then in order for both overlaps $z^3$ and $z^2y$ to resolve, suitable rescalings of both $G_1$ and $G_2$ must be appended to $\mathscr B$, implying $\dim T_3=9$. The first statement follows.

If $T$ is a twisted tensor product, then the Hilbert series of $T$ is $(1-t)^{-3}$, hence $\dim T_3=10$. 
If $G_2=0$, then $A=0$, so $(a,d)$ is not a root of any $f_n(t,u)$ by Proposition \ref{not a zero of fn}. In particular, $a\neq 1$ and thus $G_1\neq 0$. 

If $G_2\neq 0$, then $G_1=\l G_2$ for some scalar $\l$.  Since $G_2$ has no $zxz$ term, we must have $d=-1$. Since $T$ is in Jordan normal form, $e\in\{0,1\}$. We claim $e=0$. If $e=1$, then $d=E=-1$. Examining the coefficients of $zx^2$ and $x^2z$, we have $1-a=\l A$ and $A+1-a=\l A$. This implies $A=0=1-a$, contradicting Proposition \ref{not a zero of fn}. Thus $e=0$.

By Lemma \ref{JNFs}, since $e=0$, we have $A\in\{0,1\}$. We claim $A=1$. On the contrary, if $A=0$, then note that $$G_2 = E(1-B-E) yzx+E(1-B-E) yxz + S$$
where $S\in {\rm Span}_{\k} \{y^i x^j z^k\}.$ Since $T$ is a twisted tensor product, $\{y^ix^jz^k\}$ is a $\k$-basis for $T$, so the first two terms in the above expression for $G_2$ cannot vanish. Thus $E\neq 0$. Since $G_1=\l G_2$, comparing coefficients of $zx^2$ shows that $a-1=0$, again contradicting Proposition \ref{not a zero of fn}. Thus $A=1$, and the fact that $\l = 1-a$ follows by considering the coefficients of the $zx^2$ terms.
\end{proof}

In a forthcoming paper we will study the geometry, in the sense of \cite{ATVI}, of twisted tensor products. In the case $f = 1$, $G_2 = 0$, the point scheme of $T$ is reducible; whereas in the case $f = 1$, $G_2 \neq 0$, generically, the point scheme of $T$ is an elliptic curve. We refer to quadratic twisted tensor products (in Jordan normal form) where $f=1$ and $G_2=0$ as \emph{reducible}, and those where $f=1$ and $G_2\neq 0$ as \emph{elliptic}. 

We consider these two cases in the following subsections.

\subsection{The reducible case.}
In this subsection we fix $$T = T(a, b, c, d, e, 1; A, B, C, 0, E, 0)$$ and show that if $G_2=0$ and if $(a,d)$ is not a root of any $f_n(t,u)$, then $T$ is a quadratic twisted tensor product. 

We define a filtration on $T$ following Section 4.2.1 of \cite{Loday-Vallette}. Let $V_1 = \k y$ and $V_2 = \k x \oplus \k z$, so $T_1=V_1\oplus V_2$. Then the natural identification
$$(T_1)^{\tsr n} \cong \bigoplus_{(i_1,\ldots, i_n)\in \{1,2\}^n} V_{i_1}\tsr \cdots\tsr V_{i_n}$$
endows the tensor algebra $T(T_1)$ with the structure of a $\k$-vector space graded by the ordered monoid $\mathcal M=\bigcup_{n=0}^{\infty}\{1,2\}^n$ of all tuples with entries in $\{1,2\}$. The set $\mathcal M$ is a monoid under concatenation of tuples and the left-lexicographic order $$\emptyset<(1)<(2)<(1,1)<(1,2)<(2,1)<(2,2)<\cdots$$ 
is compatible with concatenation.
It is clear that the $\mathcal M$-grading respects the multiplication in $T(T_1)$. Observe that this is a refinement of the $\N$-grading given by tensor degree. 

The filtration naturally associated to this $\mathcal M$-grading is given by
$$F_{(i_1,\ldots,i_n)}T(T_1)= \bigoplus_{(j_1,\ldots,j_m)\le (i_1,\ldots,i_n)} V_{j_1}\tsr\cdots\tsr V_{j_m}.$$
The canonical projection $T(T_1)\to T$ induces a filtration on $T$, and we denote the associated $\mathcal M$-graded algebra by $\gr_F T$ or just $\gr\ T$. 

There is a canonical $\mathcal M$-graded algebra homomorphism $T(T_1)\to \gr\ T$, and we denote the kernel of its restriction to $T_1\tsr T_1$ by $R_{\rm lead}$. Let $T^{\circ} = T(T_1)/(R_{\rm lead})$. Noting that $G_2=0$ implies $A=0$, we have
$$T^{\circ} = \k \la y, x, z\ra/\la zx-ax^2-dxz-z^2, zy, xy \ra.$$

\begin{lemma} 
\label{Tcirc}
Assume $G_2=0$ and $(a,d)$ is not a root of any $f_n(t,u)$. Then the algebra $T^{\circ}$ is Koszul with Hilbert series $H_{T^{\circ}} = (1-t)^{-3}$. The set $\{y^i x^j z^k : i, j, k \geq 0\}$ is a $\k$-basis for $T^{\circ}$.
\end{lemma}

\begin{proof} 
Let $D=\k\la x,z\ra/\la  zx-ax^2-dxz-z^2\ra$ and define a graded twisting map $\p: D \tsr \k[y] \to \k[y] \tsr D$ by 
${\p}(1 \tsr y) = y\tsr 1$, ${\p}(x \tsr 1) = 1\tsr x$, ${\p}(z \tsr 1) = 1\tsr z$, and ${\p}_{\ge 2}=0$. Note that $\k[y]  \tsr_{\p} D$ and $T^{\circ}$ are isomorphic as $\N$-graded algebras.

Since $(a,d)$ is not a root of any $f_n(t,u)$, \cite[Theorem 6.2]{C-G} and our results in Section \ref{ttps of two one-variable polynomial algebras} imply that $D = \k[x] \tsr_{\s} \k[z]$ for the obvious twisting map $\s$. Thus $T^{\circ}\cong \k[y]  \tsr_{\p} (\k[x] \tsr_{\s} \k[z])$. This shows $\{y^i x^j z^k : i, j, k \geq 0\}$ is a $\k$-basis for $T^{\circ}$ and hence the Hilbert series of $T^{\circ}$ is $(1-t)^{-3}$.

Finally, by \cite[Theorem 6.2 and Theorem 5.5]{C-G} we see that $D$ is Koszul. By  \cite[Theorem 5.3 (2)]{C-G}, $\k[y] \tsr_{\p} D$ is Koszul, hence $T^{\circ}$ is Koszul. 
\end{proof}

\begin{thm}
\label{ttps with A=0, diagonal Jordan form}
If $G_2=0$ and $(a,d)$ is not a root of any $f_n(t,u)$, then $$T = T(a, b, c, d, e, 1; A, B, C, 0, E, 0)$$ is a quadratic twisted tensor product. Moreover, $T$ is Koszul.
\end{thm}

\begin{proof}
Since $(a,d)$ is not a root of $f_1(t,u) = 1-t$, $a\neq 1$ and hence $G_1\neq 0$. By Lemma \ref{cases}, $\dim T_3 = 10$. As graded vector spaces, $T$ and $\gr \ T$ are isomorphic, so $\dim (\gr\ T)_3=10$. By Lemma \ref{Tcirc}, the Hilbert series of $T^{\circ}$ is $(1-t)^{-3}$, so the canonical graded projection $T^{\circ} \to \gr \ T$ is injective in homogeneous degree 3. 

Lemma \ref{Tcirc} also shows that the algebra $T^{\circ}$ is Koszul, so, by \cite[Theorem 4.2.4]{Loday-Vallette}, the canonical graded projection $T^{\circ} \to \gr \ T$ is an isomorphism and $T$ is Koszul. 

By Lemma \ref{Tcirc}, the algebra  $T^{\circ}$ has a linear basis of the form $\{y^i x^j z^k : i, j, k \geq 0\}$. Since the projection $T^{\circ}\to \gr\ T$ is an isomorphism, $\gr \ T$ also has this basis, and consequently $T$ does as well. It follows that $T$ is a twisted tensor product. 
\end{proof}

We conclude this subsection by recording the solution set of the system of equations determined by setting $G_2=0$, and consequently, the parameter values for which $T$ is a twisted tensor product of $R$ and $S$. The following is immediate from the form of the expression $G_2$.

\begin{lemma}
\label{system for z^2y to resolve}
The expression $G_2=0$ if and only if $A=0$ and all of the following hold:
\begin{align*}
&(1) \ \ E(1-B-E) = 0,  \\ %
&(2) \ \ E(-d-B+dE) = 0, \\ %
&(3) \ \ B(1-d-B)-a(1-E^2) = 0,  \\ %
&(4) \ \ E(C+CE+e-eE) = 0, \\%
&(5) \ \ C(1-d-2B-BE)-b(1-E^2)-eB = 0, \\ %
&(6) \ \ (1+E)(-c(1-E)-C^2)-eC = 0. \\ %
\end{align*}
\end{lemma}

\begin{thm}
\label{ttps with G_2 = 0}
Suppose that $T = T(a, b, c, d, e, 1; A, B, C, 0, E, 0)$ is a quadratic twisted tensor product in Jordan normal form. If $G_2=0$, then $A = 0$, $(a, d)$ is not a zero of $f_n(t,u)$ for all $n \geq 1$, and the parameters satisfy one of the following cases:
\begin{itemize}
\item[(i)] $a = B(1-d-B)$, $b = 0$,  $c = 0$, $e=C = E = 0$;
\item[(ii)] $e=B = C = 0$, $E = 1$;
\item[(iii)] $d = -1$, $B = 2$, $e=C = 0$, $E = -1$;
\item[(iv)] $a = B(1-d-B)$, $b = 1-d-2B$, $c=-1$, $C = 1$, $e=E = 0$;
\item[(v)] $e=0$, $d= E= -1$, $B = 2$, $C = 1$.
\item[(vi)] $e=1$, $d=E=0$, $a=B(1-B)$, $b=C-B-2BC$, $c=-C(1+C)$;
\item[(vii)] $B=C=0$, $e=d=E=1$.
\end{itemize}
Conversely, if $(a, d)$ is not a zero of $f_n(t,u)$ for all $n \geq 1$, and if the parameters satisfy $A=0$ and the conditions in one of $(i)$--$(vii)$, then $T$ is a twisted tensor product of $R$ and $S$.
Moreover, all of the algebras described in $(i)$-$(vii)$ are Koszul.
\end{thm}

\begin{proof}
By Lemmas  \ref{cases} and \ref{system for z^2y to resolve}, since  $G_2=0$ we have $A=0$ and $(a, d)$ is not a zero of $f_n(t,u)$ for all $n \geq 1$. Since $T$ is assumed to be in Jordan normal form, $e\in\{0,1\}$. If $e=0$, then $C\in\{0,1\}$ and if $e=1$ then $d=E$ by Lemma \ref{JNFs}. Furthermore, equation (1) in Lemma \ref{system for z^2y to resolve} implies $E=0$ or $B+E=1$. Thus to prove the first part of the theorem, it suffices to consider the cases $(e=C=E=0)$, $(e=C=0, E\neq 0)$, $(e=E=0, C=1)$, $(e=0, C=1, E\neq 0)$, $(e=1, d=E=0)$, and $(e=1, d=E\neq 0)$.

Case (i) results from setting $e=C=E=0$ in the equations of Lemma \ref{system for z^2y to resolve}. If $e=C=0$ and $E\neq 0$, then  $B+E=1$ and equation (2) in Lemma \ref{system for z^2y to resolve} becomes $BE(1+d)=0$. Since $E\neq 0$, we have $B=0$ or $d=-1$. If $B=0$ we have $E=1$, which is case (ii). If $d=-1$, then equation (3) in Lemma \ref{system for z^2y to resolve} can be rewritten as $(1-a)(1-E^2)=0$. Since $f_1(a,d)=1-a\neq 0$, we have $E=\pm 1$. If $E=1$, the parameters belong to case (ii), otherwise we have case (iii).

Setting $e=E=0$ and $C=1$ in the equations of Lemma \ref{system for z^2y to resolve} results in case (iv). If $e=0$, $C=1$, and $E\neq 0$, then $B+E=1$ and, from equation (4) of Lemma  \ref{system for z^2y to resolve}, $1+E=0$. These two equations imply $B=2$ so, using equation (5) of Lemma \ref{system for z^2y to resolve}, we have  $d=-1$, which is case (v).

Now we turn to the case $e=1$ and $d=E$. If $d=E=0$, then the equations of Lemma \ref{system for z^2y to resolve} determine $a, b,$ and $c$ in terms of $B$ and $C$ as in case (vi). If $d=E\neq 0$, then $B+E=1$ and equation (2) of Lemma \ref{system for z^2y to resolve} becomes $E(E^2-1)=0$, so $E^2-1=0$. Using $B+E=1$ and $E^2-1=0$ in concert with equation (5) shows $B=0$, hence $E=1$ and, by equations (4) and (6), $C=0$. This is case (vii), and the first part of the proof is complete.

The converse, and the fact that $T$ is Koszul in each case, follow from Theorem \ref{ttps with A=0, diagonal Jordan form} and  Lemma \ref{system for z^2y to resolve}, after verifying that all six equations of the Lemma are satisfied in each case. 
\end{proof}

\subsection{The elliptic case.}

In this subsection we treat the case where $f = 1$ and $G_2 \neq 0$.
\begin{lemma}
\label{caseA=1}
If $T$ is a graded twisted tensor product in Jordan normal form such that $f  = 1$ and $G_2\neq 0$, then $A=1$, $d = E = -1$, $e=0$, and $b = (1-a)(2-B)$.
\end{lemma}

\begin{proof}
By Lemma \ref{cases}, we have $e=0$, $d=-1$, $A=1$, and $G_1=(1-a)G_2$.  Now $G_1$ and $G_2$ simplify to:
\begin{align*}
G_1 &= (a-1)zx^2 - (a-1) x^2z+ bx^3 +bEyzx -byxz\\
&+(bB+c+cE)yx^2+(cE^2-c)y^2z  +(bC+cB+cBE)y^2x\\
&+(cC+cCE)y^3 \\
G_2 &= -zx^2 - Ex^2z +(2-B)x^3+(E-BE-E^2)yzx\\
&+(E-BE-E^2)yxz+(2B-a-B^2-C-CE+aE^2)yx^2\\
&-(CE^2+CE)y^2z+(2C-b-2BC-BCE+bE^2)y^2x\\
 &-(c+C^2+C^2E-cE^2)y^3.\\
\end{align*} 
The fact that $b=(1-a)(2-B)$ follows by considering the coefficients of $x^3$.

Denote the coefficients of $x^2z, x^3, yzx,$ and so on, in the above expression for $G_2$ by $\a_{x^2z}, \a_{x^3}. \a_{yzx},$ etc.\ , respectively.
Since $G_1=(1-a)G_2$, adjoining $$w=zx^2-\a_{x^2z}x^2z-\a_{x^3}x^3-\a_{yzx}yzx-\a_{yxz}yxz-\a_{yx^2}yx^2-\a_{y^2z}y^2z- \a_{y^2x}y^2x-\a_{y^3}y^3$$ to the set $\mathscr B$ defined above makes both overlaps $z^3$ and $z^2y$ resolvable. So $\mathscr B'=\mathscr B\cup \{w\}$ is a Gr\"obner basis to degree 3 for $T$. Applying reductions corresponding to elements of $\mathscr B'$ to the difference $(zx^2-f_{zx^2})y-zx(xy-yx)$, we obtain
\begin{align*}
(1+E)x^4&-[2(\a_{x^3}+\a_{yzx})-B(1+E)]yx^3\\
&-[2(\a_{yx^2}+B\a_{yzx})-C(1+E)^2]y^2x^2\\
&-[2(\a_{y^2x}+C\a_{yzx})-BCE(1+E)]y^3x\\
&-[2(\a_{y^3}+C\a_{y^2z})-C^2E(1+E)]y^4.
\end{align*}
As $T$ is a twisted tensor product of $\k[x,y]$ and $\k[z]$, the obvious map $\k[x,y]\to T$ is injective. Thus the expression above, whose image in $T$ vanishes, must also vanish in the free algebra $\k\la x,y,z\ra$. Hence we have $E=-1$. 
\end{proof}

We note that the conditions $A=1$, $d=E=-1$, $e=0$ and $b=(1-a)(2-B)$ imply $G_1=(1-a)G_2$, so there are no additional restrictions on the parameters to consider. We also note the calculation used to deduce $E=-1$ also shows that the overlap $zx^2y$ created by adjoining $w$ to $\mathscr B$ is resolvable when these conditions hold. 

Now we will prove that the conditions of Lemma \ref{caseA=1} ensure that $T$ is a graded twisted tensor product.

\begin{prop}
\label{G-basis for ttp diagonal Jordan form, A=1}
Let $T = T(a, (1-a)(2-B), c, -1, 0, 1; 1, B, C, 0, -1, 0)$, so $b = (1-a)(2-B)$. Let $\b = 2-B$. The ideal of relations that defines $T$ has a finite Gr\"obner basis given by:
\begin{align*}
z^2 &- (zx-ax^2-byx-cy^2+xz) \\
zy &- (x^2+Byx+Cy^2-yz) \\
xy &- yx \\
zx^2 &- (x^2z-\b yxz+\b x^3+B\b yx^2+ C \b y^2x - \b yzx).
\end{align*} 
In particular,  $\{y^i x^j (zx)^k z^l : i, j, k \geq 0, l \in \{0,1\} \}$ is a $\k$-linear basis for $T$, and the Hilbert series of $T$ is $1/(1-t)^3$.
\end{prop}

\begin{proof}
Since the conditions $A=1$, $d=E=-1$, $e=0$ and $b=(1-a)(2-B)$ imply $G_1=(1-a)G_2$, adjoining $w=zx^2-(x^2z-\b yxz+\b x^3+B\b yx^2+ C \b y^2x - \b yzx)$ to the set $\mathscr B$ makes both overlaps $z^3$ and $z^2y$ resolvable. Two new overlaps are created: $z^2x^2$ and $zx^2y$. Straightforward calculations show these overlaps are resolvable as well. Thus $\mathscr B\cup\{w\}$ determines a finite Gr\"obner basis for $T$.

The stated basis is precisely the set of monomials that do not contain any of $xy, zy, z^2,$ or $zx^2$ as subwords. The Hilbert series follows by an easy counting argument.
\end{proof}

\begin{prop}
\label{basis for ttp diagonal Jordan form, A=1}
The set $\S = \{y^i x^j z^k : i, j, k \geq 0\}$ is a $\k$-basis for $T$.
\end{prop}

\begin{proof} Let $U$ denote the $\k$-linear span of $\S$. To show that $U = T$, by Proposition \ref{G-basis for ttp diagonal Jordan form, A=1} and the fact that $x$ and $y$ commute, it suffices to prove that for all $i \geq 0$, that $(zx)^i$ is in $U$. 

To start we recall that $zy = x^2+B yx + Cy^2-yz$. 
An easy induction shows that for all $j\ge 1$, there exist polynomials $p_j(x,y)$  such that 
$$zy^j = p_j(x,y) + (-1)^jy^jz.$$ 
In particular, $zy^j$ is in $U$ for all $j\ge 0$.


Next we claim that $z^2 x^j$ and $zx^j$ are in $U$. One checks that there exist polynomials $f_1(x,y)$ and $f_2(x,y)$ such that $$z^2x = f_1(x,y) + f_2(x,y)z^2.$$ Now for $j > 1$, $$z^2x^j = (f_1(x,y)+f_2(x,y)z^2)x^{j-1},$$ so inductively we see that $z^2x^j \in U$. Furthermore, $zx = ax^2+byx+cy^2+z^2-xz \in U$ and hence $$zx^j = (ax^2+byx+cy^2+z^2-xz)x^{j-1},$$ so inductively $zx^j \in U$.

Now we claim that $(zx)^i$ is in $U$ for all $i\ge 1$. We have noted that $zx\in U$. Suppose, inductively, that $(zx)^i \in U$ for some $i\ge 1$. Since $x$ and $y$ commute, it follows that $x(zx)^i\in U$. Write $x(zx)^i = \sum a_{jkl}y^jx^kz^l$. Then $$(zx)^{i+1} = zx(zx)^i =\sum a_{jkl}zy^jx^kz^l=\sum a_{jkl}p_j(x,y)x^kz^l+(-1)^j\sum a_{jkl}y^jzx^kz^l.$$ and it is apparent from the observations above that $(zx)^{i+1} \in U$.

The fact that the Hilbert series of $T$ is $1/(1-t)^3$ implies that $\S$ is linearly independent.

\end{proof}

\begin{thm}
\label{characterization of elliptic type ttps}
The algebra $T = T(a, (1-a)(2-B), c, -1, 0, 1; 1, B, C, 0, -1, 0)$ is a twisted tensor product of $R$ and $S$. Moreover, the generators $x, y, z$ of $T$ are left and right regular.
\end{thm}

\begin{proof}
The first statement follows from Proposition \ref{basis for ttp diagonal Jordan form, A=1} and the Hilbert series of $T$.  For the second statement, note that $T \cong T^{op}$, as graded algebras. It follows that an element of $T$ is left regular if and only if it is right regular. It is clear from the basis $\{y^i x^j z^k : i, j, k \geq 0\}$, and the fact that $x$ and $y$ commute, that $x$ and $y$ are right regular and that $z$ is left regular.

\end{proof}

We conclude this section by remarking that, up to isomorphism, all of the quadratic twisted tensor products of $R$ and $S$ are described in Theorem \ref{Ore extension classification} (Ore type), Theorem \ref{ttps with G_2 = 0} (reducible type), and Theorem \ref{characterization of elliptic type ttps} (elliptic type).

\section{The Koszul property and Yoneda algebras}

We continue to use the notation $R = \k[x,y]$, $S = \k[z]$ and $$T = T(a, b, c, d, e, f; A, B, C, D, E, F)$$ (cf. the beginning of Section \ref{sec:QuadTTPs}). In this section we determine which of the quadratic twisted tensor products of $R$ and $S$ are Koszul. In \cite{C-G} it was asked if there exist Koszul algebras $R$ and $S$, and a twisting map $\t: S \tsr R \to R \tsr S$ such that the algebra $R \tsr_{\t} S$ is quadratic, but not Koszul. Theorem \ref{Koszul property for T(g, d)} below affords examples of this phenomenon. We also compute the structure of the Yoneda algebra for these non-Koszul examples. (Recall that for a Koszul algebra, the Yoneda algebra and the quadratic dual algebra are isomorphic; see \cite[Definition 1, p. 19]{PP} for example.)

\subsection{The Koszul property} 

 In the case of an Ore-type or reducible twisted tensor product, the Koszul property follows immediately from the results in the preceding section.

\begin{thm}
\label{Koszul property for Ore extensions, A = 0}
Let $T$ be a quadratic twisted tensor product of $R$ and $S$. If $T$ is of Ore type or reducible type, then $T$ is Koszul.
\end{thm}

\begin{proof}
If $T$ is of Ore type, then $T$ is a graded Ore extension of $R$ by Proposition \ref{oneSidedOre}. Since $R$ is Koszul, it follows that $T$ is Koszul (see \cite[Chapter 4, Section 7, Example 2]{PP}, for example). 

If $T$ is of reducible type, then $T$ is Koszul by
Theorem \ref{ttps with A=0, diagonal Jordan form}. 
\end{proof}

It remains to consider the quadratic twisted tensor products where $A = 1$. Recall that these are described in Theorem \ref{characterization of elliptic type ttps}. It is convenient, for this section and the remainder of the paper, to change the presentation of the algebras of this type. This change of presentation, albeit motivated by easing computations, is also natural in a certain sense: the cubic equation defining the point scheme (generically an elliptic curve) is in Weierstrass form (see, for example, \cite[Chapter III.1]{Sil}). In order to make this change of presentation it is necessary to assume that $\text{char}\, {\k} \ne 2$. We note that our main results, Theorem \ref{Koszul property for T(g, d)} and Theorem \ref{AS-regular algebras}(3) below, can be proved by the exact same methods without changing presentation, hence they can be seen to hold over any field. 

\begin{lemma}
\label{simple presentation, type A=1}
Suppose that ${\rm char}\, {\k} \ne 2$. 
Define $$\b = 2-B, \ \ \g = C+2(a-1), \ \ g = \g-\b^2/4, \ \ h = c-(a-1)(C+a-1).$$ The algebra $T(a, (1-a)(2-B), c, -1, 0, 1; 1, B, C, 0, -1, 0)$ can be presented as $$\frac{\k \la x, y, w \ra}{\la wy+yw-x^2-gy^2, w^2+h y^2, xy-yx \ra}.$$
\end{lemma}

\begin{proof}
This is a straightforward computation using the invertible change of variables: $x \mapsto x-(\b/2)y$, $y \mapsto y$, $w \mapsto -x+(a-1)y+z$.
\end{proof}

Using this lemma, we change notation and write $$T(g, h) = T(a, (1-a)(2-B), c, -1, 0, 1; 1, B, C, 0, -1, 0).$$ When we use this notation, we are implicitly assuming that $\text{char} \, {\k} \ne 2$.
Below we will establish when the Koszul property holds for $T(g,h)$ by computing a minimal graded free resolution of the trivial module $_T\k$. That calculation makes use of the Gr\"obner basis described in the next lemma.

\begin{lemma}
\label{Groebner basis for T(g, d)}
Order the generators of $T(g, h)$ as $y < x < w$ and use left lexicographical ordering on the monomials in $\k \la y, x, w \ra$. Then the defining ideal of $T(g, h)$ has a finite Gr\"obner basis consisting of
\begin{align*}
&xy-yx \\
&wy+yw-x^2-gy^2 \\
&w^2+h y^2 \\
&wx^2 -x^2w.
\end{align*} 
Consequently, $\{y^i x^j (wx)^k w^l : i, j, k \geq 0, l \in \{0, 1\}\}$ is a $\k$-basis for $T(g, h)$. The elements $x^2$ and $y^2$ are central in $T(g, h)$.
\end{lemma}

\begin{proof}
The proof of the first statement is a straightforward computation using Bergman's diamond lemma. The second and third statements follow immediately from the Gr\"obner basis.
\end{proof}

For later reference, we fix the following sequences of graded free modules, which we will prove are resolutions of $T/T_+={_T\k}$.

\begin{defn}
\label{resolutions}
Let $T=T(g,h)$. 
\begin{enumerate}
\item If $h\neq 0$, define $(Q_{\bullet},d^Q_{\bullet})$ to be the sequence 
$$0\to T(-3) \xrightarrow{d^Q_3} T(-2)^{3} \xrightarrow{d^Q_2} T(-1)^{3} \xrightarrow{d^Q_1} T,$$ where $$d^Q_3 = \begin{bmatrix} h y & w & -h x \end{bmatrix}, \ \ d^Q_2 = \begin{bmatrix} -x & w-gy & y \\ 0 & h y & w \\ -y & x & 0 \end{bmatrix}, \ \ d^Q_1 = \begin{bmatrix} x \\ y \\ w \end{bmatrix}.$$
\item If $h=0$, define for each $i \geq 0$ a graded free left $T$-module by $$P_i = \begin{cases} T & i = 0 \\ T(-1)^3 & i = 1 \\ T(-2)^3 & i = 2 \\ T(-i) \oplus T(-i-1) & i \geq 3. \end{cases}$$ Also, define a map $d^P_i: P_i \to P_{i-1}$ via 

$$d^P_1 = \begin{bmatrix} x \\ y \\ w \end{bmatrix}, \ \ d^P_2 = \begin{bmatrix} -x & w-g y & y \\ 0 & 0 & w \\ -y & x & 0 \end{bmatrix}, \ \ d^P_3 = \begin{bmatrix} 0 & w & 0 \\ wy & -y^2 & -wx \end{bmatrix}, \ \ d^P_4 = \begin{bmatrix} w & 0 \\ y^2 & w \end{bmatrix},$$ and for all $i \geq 5$, $$d^P_i = \begin{bmatrix} w & 0 \\ y^2 & -w \end{bmatrix}.$$
\end{enumerate}

\end{defn}

Using Lemma \ref{Groebner basis for T(g, d)}, it is easy to check that $(Q_{\bullet},d^Q_{\bullet})$ and $(P_{\bullet},d^P_{\bullet})$ are complexes of graded free modules with $\coker\ d^Q_1$ and $\coker\ d^P_1$ isomorphic to $T/T_+={_T \k}$.

The following standard fact is very useful for proving exactness of complexes of locally-finite graded modules.

\begin{lemma}
\label{dimSum}
Let $C_{\bullet}:0 \to V_n \to V_{n-1} \to \cdots \to V_1 \to V_0 \to 0$ be a finite complex where the $V_i$ are finite-dimensional vector spaces. Then
\[\sum_{i = 0}^n (-1)^i \dim(V_i) = \sum_{i = 0}^n (-1)^i \dim(H_i(C_{\bullet})). \]
\end{lemma}

Now we are ready to use Lemma \ref{Groebner basis for T(g, d)} to prove the main theorem of this section. 

\begin{thm}
\label{Koszul property for T(g, d)} Let $T=T(g,h)$.
The complexes $(Q_{\bullet},d^Q_{\bullet})$ and $(P_{\bullet},d^P_{\bullet})$ are graded free $T$-module resolutions of $_{T}\k$ when $h\neq 0$ and $h=0$, respectively. In particular, the algebra $T(g, h)$ is Koszul if and only if $h = c-(a-1)(C+a-1)$ is nonzero.
\end{thm}

\begin{proof}
Since each $Q_i$ is generated in degree $i$, but $P_i$ is generated in degrees $i$ and $i+1$ for $i\ge 3$, it suffices to prove that $(Q_{\bullet},d^Q_{\bullet})$ and $(P_{\bullet},d^P_{\bullet})$ are exact in homological degrees $>0$.  

Assume that $h \ne 0$. We claim that the map $d^Q_3$ is injective. Suppose that $t \in \ker(d^Q_3)$. Then $tx = ty = tw = 0$, so $tT_1=0$. Then Theorem \ref{characterization of elliptic type ttps} implies that $t = 0$, as desired. We have shown that the complex $$0 \xrightarrow{} T(-3) \xrightarrow{d^Q_3} T(-2)^{3} \xrightarrow{d^Q_2} T(-1)^{3} \xrightarrow{d^Q_1} T$$ is exact at $T(-3)$; since $T$ is quadratic,  the complex is exact at $T(-1)^3$. Because the Hilbert series of $T$ is $(1-t)^{-3}$, it follows that the complex is exact. Since $_T\k$ has a linear free resolution, we conclude that $T$ is Koszul when $h \ne 0$.

Now assume that $h = 0$. Notice that $w^2 = 0$ in $T$ so the map $\begin{bmatrix} 0 & w & 0 \end{bmatrix}$ is not injective. Since $T$ is quadratic, the complex $(P_{\bullet}, d^P_{\bullet})$ is exact at $T(-1)^3$.

Now we prove exactness at $T(-3)$. Suppose that $\begin{bmatrix} u_1 & u_2 \end{bmatrix} \in \ker(d^P_3)$. Then it follows that $u_2 wx = 0$. From the $\k$-basis described in Lemma \ref{Groebner basis for T(g, d)} we see that $u_2 = rw$ for some $r \in T$. We also have $u_1 w - u_2 y^2 = 0$, so, using the fact that $y^2$ is central, we have $(u_1-ry^2)w = 0$. The left annihilator ideal of $w$ is $Tw$, so $u_1-ry^2 = uw$ for some $u \in T$. Therefore $$\begin{bmatrix} u_1 & u_2 \end{bmatrix} = \begin{bmatrix} ry^2+uw & rw \end{bmatrix}  = u \begin{bmatrix} w & 0 \end{bmatrix} +r \begin{bmatrix} y^2 & w \end{bmatrix} \in \im(d^P_4).$$ We conclude exactness at $T(-3)$.

A similar, and easier, argument is used to show exactness in homological degree $i$ for all $i \geq 4$. We leave this argument to the reader.

To show exactness at $T(-2)^3$, note that for any fixed homogeneous degree $j$, the complex $(P_{\bullet},d^P_{\bullet})$ restricts to a finite complex of finite-dimensional vector spaces. Let $q(t)$ denote the Hilbert series of the homology group $H_2(P_{\bullet})$, and let $H_T$ denote the Hilbert series of $T$.  Then the Hilbert series of $T(-i)$ is $t^iH_T$. Appending $T \to {_T\k} \to 0$ to $(P_{\bullet},d^P_{\bullet})$ and applying Lemma \ref{dimSum} yields $$1-H_T+3tH_T -3t^2 H_T + t^3 H_T = -q(t).$$ Rearranging and using the fact that $H_T = (1-t)^{-3}$ we have 
$$1+q(t) = H_T(1-3t+3t^2-t^3) = 1,$$ so $q(t) = 0$. We conclude exactness of $P_{\bullet}$ at  $T(-2)^3$. Thus, $(P_{\bullet}, d^P_{\bullet})$ is a graded free resolution of $_T \k$.


\end{proof}


\subsection{The Yoneda algebra of $T(g, h)$}

For any graded algebra $A$, the graded Hom functor for graded left $A$-modules is 
$$\Hom_A(M,N)=\bigoplus_{n\in\Z} \Hom^n_A(M,N)=\bigoplus_{n\in\Z} \hom_A(M,N[n]),$$ 
where $\hom_A(M,N[n])$ is the space of degree-0 graded $A$-module homomorphisms $M\to N[n]$. This functor is left exact, and its $i$-th right derived functor is denoted $\Ext^i_A(M,N)$. The space $\Ext^i_A(M,N)$ inherits a grading from the graded Hom functor, and we denote the homogeneous degree-$j$ component by $\Ext^{i,j}_A(M,N)$.

One may compute the space $\Ext^i_A(M,N)$ as the homology group $H^i(P_{\bullet},N)$ where $(P_{\bullet},d_{\bullet})$ is a graded free resolution of $M$. Thus 
$$\Ext_A(M,N)=\bigoplus_i \Ext^i_A(M,N)=\bigoplus_{i,j} \Ext^{i,j}_A(M,N)$$ is bigraded.
When $M=N={_A}\k$, we abbreviate $E^{i,j}(A)=\Ext^{i,j}_A(\k,\k)$, $E^i(A)=\bigoplus_j E^{i,j}(A)$ and $E(A)=\bigoplus_i E^i(A)$. The vector space $E(A)=\Ext_A(\k,\k)$ admits the structure of a bigraded algebra called the \emph{Yoneda algebra} via the Yoneda composition product. We briefly recall the definition of this product, see \cite[p. 4]{PP} for more details. 

Let $\a \in E^i(A)$ and $\b \in E^j(A)$. The product $\a \star \b \in E^{i+j}(A)$ is defined as follows. Choose representatives $f \in \Hom_A(P_i, \k)$ and $g \in \Hom_A(P_j, \k)$ for the classes $\a$ and $\b$, respectively. Using projectivity, one lifts $g$ as in the first two rows of the following diagram to obtain a map $g_i \in \Hom_A(P_{i+j}, P_i)$. Then $\a \star \b$ is defined to be $[f \circ g_i] \in E(A)^{i+j}$.

\centerline{\xymatrix{
P_{i+j} \ar[r] \ar[d]^{g_i} & P_{i+j-1} \ar[r] \ar[d]^{g_{i-1}} & \cdots \ar[r] & P_j \ar[d]^{g_0} \ar[rd] ^{g}\\
P_i \ar[r] \ar[d]^{f} & P_{i-1} \ar[r] & \cdots \ar[r] &P_0 \ar[r]_{\e} & \k \\
\k
}}

We conclude this section by establishing a presentation for the Yoneda algebra of $T(g, h)$. This family of examples illustrates that, in general, relationships between the Yoneda algebras of $R$, $S$ and $R \tsr_{\t} S$ may not be simple or obvious.

Let $T = T(g, h)$. Let $(Q_{\bullet},d^Q_{\bullet})$ and $(P_{\bullet},d^P_{\bullet})$ be the complexes of graded free $T$-modules constructed in Definition \ref{resolutions}. Recall that by Theorem \ref{Koszul property for T(g, d)}, these are graded free resolutions of $_T \k$ when $h\neq 0$ and $h=0$, respectively. Since $\im\ d^Q_i\subset T_+Q_{i-1}$ and  $\im\ d^P_i\subset T_+P_{i-1}$, these resolutions are \emph{minimal}: all differentials in the complexes $\Hom_T(Q_{\bullet},\k)$ and $\Hom_T(P_{\bullet},\k)$ are trivial. Thus by Theorem \ref{Koszul property for T(g, d)}, we may take $E^i(T) = \Hom_T(Q_i, \k)$ and $E^i(T)=\Hom_T(P_i, \k)$ as our models of $\Ext^i_T(\k, \k)$, according to whether or not $h=0$. 

Let $\x, \n, \w \in E^{1,1}(T)$ denote dual basis vectors to $x, y, w \in T_1$, respectively. In the case $h=0$, let $\p \in E^{3,4}(T)$ denote the graded $T$-linear map: $\p: T(-3) \oplus T(-4) \to \k$ given by $\p(1, 0) = 0$, $\p(0, 1) = 1$.

\begin{thm}\label{Yoneda presentations}
Retain the notation of the previous paragraph. Then the Yoneda algebra $E(T)$ of $T(g,h)$ can be presented as follows.
\begin{itemize}
\item[(1)] If $h \ne 0$, then $E(T)$ is generated by $\x, \n, \w$ subject to the quadratic relations: $$\x \n+\n \x, \x \w, \w \x, \w \n-\n \w, \n \w+\x^2, \n^2 - h \w^2-g \x^2.$$
\item[(2)] If $h = 0$, then $E(T)$ is generated by $\x, \n, \w, \p$ subject to the quadratic relations: $$\x \n+\n \x, \x \w, \w \x, \w \n-\n \w, \n \w+\x^2, \n^2-g \x^2,$$ quartic relations: $$\x \p, \n \p, \p \x, \p \n, \w \p+\p \w,$$
and one sextic relation: $\p^2$.

Moreover, for all $i \geq 3$, $E^{i, i}(T) = \k \w^i$ and $E^{i, i+1}(T) = \k \w^{i-3} \p$.
\end{itemize}
\end{thm}

\begin{proof}
By Theorem \ref{Koszul property for T(g, d)}, $T$ is Koszul if $h \ne 0$. In this case it is well known (see \cite[Chapter 2, Definition 1 (c)]{PP}, for example) that $E(T)$ is isomorphic to the quadratic dual algebra, $T^!$. The algebra $T^!$ is generated by the space $T_1^*$, and its defining (quadratic) relations are those elements of $T_1^*\tsr T_1^*$ orthogonal to all quadratic relations of $T$ under the natural pairing. Thus statement (1) follows by checking orthogonality of the stated relations with those of $T$, and a simple dimension count.

Now we prove (2). Assume that $h = 0$, and set $T = T(g, 0)$. It follows from \cite[Proposition 3.1, p. 7]{PP} that the diagonal subalgebra, $\bigoplus_i E^{i,i}(T)\subset E(T)$ is isomorphic to the quadratic dual algebra, $T^!$. The quadratic relations of $T^!$, in this case, are obtained from the quadratic relations in (1) by setting $h = 0$. 

In order to prove that the quartic expressions given in (2) are in fact relations of $E(T)$, we use the model $E^i(T) = \Hom_T(P_i, \k)$ for $\Ext^i_T(\k, \k)$, where $(P_{\bullet},d^P_{\bullet})$ is the graded free resolution of $_T \k$ constructed in Definition \ref{resolutions}. The differentials are written in terms of matrices with respect to the natural canonical bases for the free modules appearing in the resolution. Let $\xi \in E^{4, 5}(T)$ denote the graded $T$-linear map: $\xi: T(-4) \oplus T(-5) \to \k$ given by $\xi(1,0) = 0$, $\xi(0,1) = 1$. With respect to these bases, we have $$\x = \begin{bmatrix} 1 \\ 0 \\ 0 \end{bmatrix}, \ \ \n = \begin{bmatrix} 0 \\ 1 \\ 0 \end{bmatrix}, \ \ \w = \begin{bmatrix} 0 \\ 0 \\ 1 \end{bmatrix}, \ \ \p = \begin{bmatrix} 0 \\ 1 \end{bmatrix}, \ \ \xi = \begin{bmatrix} 0 \\ 1 \end{bmatrix}.$$

Now we check that $\p \w + \w \p = 0$ in $E(T)$.   We will show that $\w \p = \xi$ and $\p \w = -\xi$. 

In order to compute $\w \p$, we must lift $\p$ as in the first two rows of the following diagram. 

\centerline{\xymatrix{
\cdots \ar[r] & P_7 \ar@{.>}[d]^{\p_4} \ar[r]^{d_7^P} & P_6 \ar@{.>}[d]^{\p_3} \ar[r]^{d_6^P} & P_5 \ar@{.>}[d]^{\p_2} \ar[r]^{d_5^P} & P_4 \ar@{.>}[d]^{\p_1} \ar[r]^{d_4^P} & P_3 \ar@{.>}[d]^{\p_0}  \ar[rd]^{\p}\\
\cdots \ar[r] & P_4 \ar[r]^{d_4^P} & P_3 \ar[r]^{d_3^P} & P_2 \ar[r]^{d_2^P} & P_1 \ar[r]^{d_1^P} \ar[d]^{\w} & P_0 \ar[r]_{\e} & \k \\
& & & & \k
}}

Let $$\p_0 = \begin{bmatrix} 0 \\ 1 \end{bmatrix}, \ \ \p_1 = \begin{bmatrix} 0 & 0 & 0 \\ 0 & 0 & 1 \end{bmatrix}, \ \ \p_2 = \begin{bmatrix} 0 & 0 & 0 \\ 0 & -1 & 0 \end{bmatrix},$$ and $$ \p_i = \begin{bmatrix} 0 & 0 \\ (-1)^{i+1} & 0 \end{bmatrix}, \text{ for all } i \geq 3.$$ Then it is straightforward to check that these maps do indeed complete the above diagram. The map $\w \circ \p_1$ is given by the matrix $\begin{bmatrix} 0 & 1 \end{bmatrix}^t$, so we conclude that $ \w \p = \xi$.

Now we compute $ \p \w$. We must lift $\w$ through the first two rows of the following diagram.

\centerline{\xymatrix{
\cdots \ar[r] & P_5 \ar[r]^{d_5^P} \ar@{.>}[d]^{\w_4} & P_4 \ar[r]^{d_4^P} \ar@{.>}[d]^{\w_3} & P_3 \ar[r]^{\ \ \ d_3^P} \ar@{.>}[d]^{\w_2} & P_2 \ar[r]^{d_2^P} \ar@{.>}[d]^{\w_1} & P_1 \ar[rd]^{\w} \ar@{.>}[d]^{\w_0} \\
\cdots \ar[r] & P_4 \ar[r]^{d_4^P} & P_3 \ar[r]^{d_3^P} \ar[d]^{\p} & P_2 \ar[r]^{ \ \ \ d_2^P} & P_1 \ar[r]^{d_1^P} & T \ar[r]^{\e} & \k \\
& & \k
}}

Let $$\w_ 0 = \begin{bmatrix} 0 \\ 0 \\ 1 \end{bmatrix}, \ \ \w_ 1 = \begin{bmatrix} 0 & 1 & 0 \\ 0 & 0 & 1 \\ 0 & 0 & 0 \end{bmatrix},  \ \ \w_ 2 = \begin{bmatrix} 0 & 1 & 0 \\ -y & 0 & x \end{bmatrix}, \ \  \w_ 3 = \begin{bmatrix} 1 & 0 \\ 0 & -1 \end{bmatrix},$$ and $$ \w_i = \begin{bmatrix} 1 & 0 \\ 0 & 1 \end{bmatrix}, \text{for all } i \geq 4.$$ 

One checks that these maps complete the last diagram. The map $\p \circ \w_3$ is given by the matrix $\begin{bmatrix} 0 & -1 \end{bmatrix}^t$, so we conclude that $\p \w = - \xi$. It follows that $\p \w + \w \p = 0$ in $E(T)$, as claimed. The other quartic relations can be proved via similar computations.

Note that because the Yoneda algebra is bigraded, $\p^2 \in E^{6, 8}(T)$. The minimal resolution $(P_{\bullet},d^P_{\bullet})$ constructed in Definition \ref{resolutions} has $P_6 = T(-6) \oplus T(-7)$, so $E^{6, 8}(T) = 0$. Hence $\p^2 = 0$ in $E(T)$. 

Now we prove the last statement of (2). Using the definitions of the maps $\p_i, \w_i$ given above, it is easy to compute that for all $i \geq 3$, $\w^i \in E^{i, i}(T)$ is given by the matrix $\begin{bmatrix} 1 & 0 \end{bmatrix}^t$;  and for all $j \geq 1$, $\w^j \p \in E^{j+3, j+4}(T)$ is given by the matrix $\begin{bmatrix} 0 & (-1)^{j-1} \end{bmatrix}^t$. Therefore, since we know by Definition \ref{resolutions} (2) that $E^{i, i}(T)$ and $E^{j+3, j+4}(T)$ are both 1-dimensional, we have $E^{i, i}(T) = \k \w^i$ for all $i \geq 3$ and $E^{j+3, j+4}(T) = \k \w^j \p$ for all $j \geq 0$.

Using the minimal resolution $(P_{\bullet},d^P_{\bullet})$, we know the Hilbert series of $E(T)$. Then a straightforward Gr\"obner-basis argument shows that the algebra with generators $\x, \n, \w, \p$ and defining relations as in (2) has the same Hilbert series as $E(T)$. Hence, the relations given in (2) are in fact a complete set of defining relations for $E(T)$.


\end{proof}

When $T(g, h)$ is Koszul, it is natural to ask if $E(T)$ is isomorphic as a $\k$-algebra to a twisted tensor product of the Yoneda algebras $E(R)$ and $E(S)$ since, as graded vector spaces, we have $E(T) =  E(R) \tsr E(S)$.

\begin{prop} If $T(g, h)$ is Koszul, then $E(T(g,h))$ is not isomorphic as a $\k$-algebra to a twisted tensor product of $E(R)$ and $E(S)$.
\end{prop}

\begin{proof}
Let $T = T(g, h)$ and assume $T$ is Koszul. Suppose, to the contrary, that $E(T)$ is isomorphic to a twisted tensor product of $E(R)$ and $E(S)$. By the standard fact that the polynomial ring $R$ is a Koszul algebra, $E(R)$ is isomorphic to the exterior algebra $\Lambda(x, y) = \k \la x, y \ra/ \la x^2, y^2, xy+yx \ra$. It follows that $E(T)$ has a subalgebra isomorphic to $\Lambda(x, y)$. Therefore there exist linearly independent elements $\a, \b \in E^1(T)$, such that $\a^2 = \b^2 = \a\b + \b \a = 0$. 

To rule out the existence of such elements, thereby obtaining a contradiction, we use the presentation for $E(T)$ given in Theorem \ref{Yoneda presentations} (1). Order the generators by $\w < \x < \n$, and order monomials in the free algebra $\k \la \x, \n, \w \ra$ by degree and left-lexicographic order. Then the associated Gr\"obner basis in degree 2 is given by: $$\n \x+\x \n, \x \w, \w \x, \n \w - \w \n, \x^2 + \w \n, \n^2-g \x^2-h \w^2.$$ Consider an element $\g = a \x + b \w \ne 0$, where $a, b \in \k$. Then, in terms of the monomial basis determined by the Gr\"obner basis, $$\g^2 = -a^2 \w \n + b^2 \w^2,$$ so $\g^2 \ne 0$. 

Write $\a = a_{\w} \w + a_{\x} \x + a_{\n} \n$ and $\b = b_{\w} \w + b_{\x} \x + b_{\n} \n$ for some $a_{\w}, a_{\x}, a_{\n}, b_{\w}, b_{\x}, b_{\n} \in \k$. It follows from the last paragraph, that $a_{\n}$ and $b_{\n}$ are both nonzero. Without loss of generality, by multiplying $\a$ and $\b$ by appropriate scalars, we may assume that $a_{\n} = b_{\n} = 1$. Since $\a$ and $\b$ are linearly independent, $\a - \b \ne 0$, and the $\n$ component of $\a - \b$ is zero. Thus, $(\a - \b)^2 \ne 0$. However, computing directly and using our assumptions, $$(\a - \b)^2 = \a^2 - \a \b - \b \a -\b^2 = 0,$$ a contradiction.

Hence, $E(T)$ is not isomorphic to a twisted tensor product of $E(R)$ and $E(S)$.
\end{proof}

\begin{rmk}
\label{regrading not necessarily Koszul}
It is also natural to ask if the algebra $E(T(g, 0))$ becomes Koszul after regrading: $\deg(\p) = 1$. However, this is not the case. For example, if $g \in \{0, 1\}$, then, after regrading, $E(T(g, 0))$ is not Koszul. 
Nevertheless, further study of $E(T(g,0))$ seems interesting.
\end{rmk}

\section{Artin-Schelter regularity}

We continue to use the notation $R = \k[x,y]$ and $S = \k[z]$. In this section we determine when a quadratic twisted tensor product of $R$ and $S$ is Artin-Schelter regular. This notion was introduced by Artin and Schelter in \cite{AS}. 

\begin{defn} 
\label{AS-regular}
\cite{AS}
A finitely-presented graded algebra, $A$, generated in degree 1 is called \emph{AS-regular} of dimension $d$ if (i) $\gldim(A) = d < \infty$; (ii) $\GKdim(A) < \infty$; and (iii) $A$ is \emph{Gorenstein}: $\Ext^n_A(\k, A) = 0$ if $n \ne d$, and $\Ext^d_A(\k, A) \cong \k$.
\end{defn}

Recall that every quadratic twisted tensor product of $R$ and $S$ is isomorphic as a twisted tensor product of $R$ and $S$ to one of the algebras described in Theorem \ref{Ore extension classification} (Ore type), Theorem \ref{ttps with G_2 = 0} (reducible type), and Theorem \ref{characterization of elliptic type ttps} (elliptic type). The property of AS-regularity is an algebra-isomorphism invariant, so the following result completely determines when a quadratic twisted tensor product of $R$ and $S$ is AS-regular.

\begin{thm}
\label{AS-regular algebras}
Let $T$ denote a quadratic twisted tensor product of $R$ and $S$.
\begin{itemize}
\item[(1)] If $T$ is an algebra of Ore type, then $T$ is AS-regular if and only if $T\cong R[z;\s,\d]$ where $\s \in \End(R)$ is invertible. 
\item[(2)] If $T$ is an algebra of reducible type, then $T$ is AS-regular if and only if $E \ne 0$ and $a+d \ne 0$.
\item[(3)] Assume that ${\rm char}\, \k \ne 2$. If $T$ is an algebra of elliptic type, then $T$ is AS-regular if and only if $h = c-(a-1)(C+a-1) \ne 0$.
\end{itemize}
\end{thm}

\begin{proof}
Since $R$ and $S$ are finitely-presented and generated in degree 1, the same is true of $T$.

Let us begin by showing that an algebra in the context of (1) or (2) has global dimension equal to $3$. If $T$ is an algebra of Ore type or reducible type, then  $T$ is Koszul by Theorem \ref{Koszul property for Ore extensions, A = 0}. Then the fact that the Hilbert series of $T$ is $(1-t)^{-3}$ implies that $\gldim(T) = 3$. 

Let $T$ denote an algebra of Ore type. By Proposition \ref{oneSidedOre},  $T$ is an Ore extension of the form $R[z; \s, \d]$. If $\s$ is not invertible, then $T$ is not a domain. It is known that every  AS-regular algebra of dimension 3 is a domain (see \cite[Theorem 8.1]{ATVI} and \cite[Theorem 3.9]{ATVII}), so if $\s$ is not invertible, then $T$ is not AS-regular. Conversely, by \cite[Proposition 2]{AST}, if $\s$ is invertible, then $T$ is AS-regular. 

Now we prove (2). Let $T$ be an algebra of reducible type. 
First we show that the conditions $E  \ne 0$ and $a+d \ne 0$ are necessary for $T$ to be AS-regular. If $E = 0$, then the equation $(z-Bx-Cy)y = 0$ holds in $T$,  showing that $T$ is not a domain. It follows that $T$ is not AS-regular when $E = 0$. 

Now assume that $E \ne 0$. Then Proposition \ref{not a zero of fn} implies that $y \in T$ is normal, and the quotient algebra $T/\la y \ra$ is isomorphic to $\k \la x, z \ra/\la z^2-zx+dxz + ax^2 \ra$. Notice that if $a+d = 0$, we have $$z^2-zx+dxz +ax^2 = (z-ax)(z-x).$$ By Proposition \ref{not a zero of fn}, we have $a \ne 1$ since $a$ is not a zero of $f_1 = 1-t$. Therefore the algebra $T/\la y \ra$ is isomorphic to $\k \la x, y \ra/ \la xy \ra$, which is not noetherian. Hence $T$ is not noetherian. As shown in \cite[Theorem 8.1]{ATVI}, all AS-regular algebras of global dimension 3 are noetherian. Consequently, if $a+d = 0$, then $T$ is not AS-regular. 

Finally, we show that $E \ne 0$ and $a+d \ne 0$ are sufficient conditions for $T$ to be AS-regular. Suppose that $E \ne 0 \ne a+d$. Then Proposition \ref{not a zero of fn} ensures that $y \in T$ is normal, and using Theorem \ref{two dim up to isom}, the quotient algebra $T/\la y \ra$ is noetherian. It follows  that $T$ is noetherian (see, for example, \cite[Lemma 8.2]{ATVI}). Then a result of Stephenson and Zhang, \cite[Corollary 0.2]{StZ}, implies that $T$ is AS-regular.

For (3), we use the notation $T(g, h)$ introduced in Section 5 for the algebra $T$. It is well known (see \cite[Theorem 2.2]{Shelton-Tingey}, for example) that quadratic AS-regular algebras are Koszul, so by Theorem \ref{Koszul property for T(g, d)}, if $h=0$, then $T$ is not AS-regular. Now suppose that $h \ne 0$. Recall from Theorem \ref{Koszul property for T(g, d)} that $$0 \xrightarrow{} T(-3) \xrightarrow{d^Q_3} T(-2)^{3} \xrightarrow{d^Q_2} T(-1)^{3} \xrightarrow{d^Q_1} T,$$ where $$d^Q_3 = \begin{bmatrix} h y & w & -h x \end{bmatrix}, \ \ d^Q_2 = \begin{bmatrix} -x & w-gy & y \\ 0 & h y & w \\ -y & x & 0 \end{bmatrix}, \ \ d^Q_1 = \begin{bmatrix} x \\ y \\ w \end{bmatrix},$$ is a minimal resolution of $_T \k$ by graded free $T$-modules. In particular, the global dimension of $T$ is $3$.

Since the Hilbert series of $T$ is $(1-t)^{-3}$, we know that the Gelfand-Kirillov dimension of $T$ is $3$. Hence to show that $T$ is AS-regular, we only need to verify the Gorenstein condition; it is clear that this condition is equivalent to the exactness of the dual complex of graded right $T$-modules

$$0 \to T \xrightarrow{d^Q_1} T(1)^3 \xrightarrow{d^Q_2} T(2)^3 \xrightarrow{d^Q_3} T(3) \to \k_T(3) \to 0.$$

Since $h \ne 0$, this complex is exact at $T(3)$. It is straightforward to check that in the tensor algebra $T \la x, y, w \ra$, the entries of $d^Q_3 d^Q_2$ give a basis for the space of defining relations of $T$. Thus the complex is exact at $T(2)^3$. By Theorem \ref{characterization of elliptic type ttps}, we know that $x$ is left regular, so the complex is exact at $T$. Finally, using Lemma \ref{dimSum} and the fact that the Hilbert series of $T$ is $(1-t)^{-3}$ (as in the proof of Theorem \ref{Koszul property for T(g, d)})  the complex is exact at $T(1)^3$.

\end{proof}

\noindent {\bf Acknowledgement.} We are very grateful to the anonymous referee whose numerous helpful suggestions greatly improved the quality of this paper.

\bibliographystyle{plain}
\bibliography{bibliog2}

\end{document}